\documentclass{article}

\usepackage{amsmath,amsthm,amssymb}
\usepackage{graphicx, hyperref, fullpage, color}
\usepackage{mathtools}
\usepackage[colorinlistoftodos,bordercolor=orange,backgroundcolor=orange!20,linecolor=orange,textsize=scriptsize]{todonotes}
\usepackage{algorithmic}
\usepackage{algorithm}
\usepackage{caption, subcaption}
\usepackage{array}

\newcommand{\R}{\mathbb{R}}
\newcommand{\E}{\mathbb{E}}

\newcommand{\eqdef}{~{:=}~}

\newcommand{\bX}{\mathbf{X}}
\newcommand{\bP}{\mathbf{P}}
\newcommand{\Prob}{\mathbb{P}}

\newcommand{\bQ}{\mathbf{Q}}
\newcommand{\bM}{\mathbf{M}}
\newcommand{\bI}{\mathbf{I}}
\newcommand{\Diag}{{\rm Diag}}

\newcommand{\CP}{C_P}
\newcommand{\CD}{C_D}

\newcommand{\BB}{\mathbb{B}^{d\times n}_{\neq 0}}

\newcommand{\mywidth}{0.25}

\DeclareMathOperator*{\argmin}{arg\,min}
\DeclareMathOperator*{\argmax}{arg\,max}

\theoremstyle{plain}
\newtheorem{theorem}{Theorem}

\newtheorem{corollary}[theorem]{Corollary}
\newtheorem{lemma}[theorem]{Lemma}

\usepackage[utf8]{inputenc} 
\usepackage[T1]{fontenc}    
\usepackage{hyperref}       
\usepackage{url}            
\usepackage{booktabs}       
\usepackage{nicefrac}       
\usepackage{microtype}      

\title{Coordinate Descent Face-Off: Primal or Dual?}

\author{
  Dominik Csiba  
\qquad 
   Peter Richt\'{a}rik\thanks{This author would like to acknowledge support from the EPSRC Fellowship EP/N005538/1, ``Randomized Algorithms for Extreme Convex Optimization''.} \\ \\
   {\em School of Mathematics}\\
   {\em University of Edinburgh}\\
   {\em United Kingdom}
}

\begin{document}

\maketitle

\begin{abstract}
Randomized coordinate descent (RCD) methods are state-of-the-art algorithms for training linear predictors via minimizing regularized empirical risk. When the number of examples ($n$) is much larger than the number of features ($d$), a common strategy is to apply RCD to the dual problem. On the other hand, when the number of features is much larger than the number of examples, it makes sense to apply RCD directly to the primal problem. In this paper we provide the first joint study of these two approaches when applied to L2-regularized ERM. First, we show through a rigorous analysis that for dense data, the above intuition is precisely correct. However, we find that for sparse and structured data,  primal RCD can significantly outperform dual RCD even if $d\ll n$, and vice versa, dual RCD can be much faster than primal RCD even if $n\ll d$. Moreover, we show that, surprisingly, a single sampling strategy minimizes both the (bound on the) number of iterations and the overall expected complexity of RCD. Note that the latter complexity measure also takes into account  the average cost of the iterations, which depends on the structure and sparsity of the data, and on the sampling strategy employed.  We confirm our theoretical predictions using extensive experiments with both synthetic and real data sets. 
\end{abstract}

\section{Introduction}

In the last 5 years or so, randomized coordinate descent (RCD) methods  \cite{ShalevTewari11, Nesterov:2010RCDM,  UCDC, PCDM} have  become immensely  popular in a variety of machine learning  tasks, with supervised learning being a prime example. The main reasons behind the rise of RCD-type methods is that they can be easily  implemented, have intuitive appeal, and enjoy superior theoretical and practical behaviour when compared to classical methods such as  SGD \cite{robbins1951}, especially in high dimensions, and in situations when solutions of medium to high accuracy are needed. One of the most important success stories of RCD is in the domain of training linear predictors via regularized empirical risk minimization (ERM).

The highly popular SDCA algorithm \cite{SDCA} arises as the application of RCD \cite{UCDC} to the {\em dual problem} associated with the (primal) ERM problem\footnote{Indeed, the analysis of SDCA in \cite{SDCA} proceeds by applying the complexity result from \cite{UCDC} to the {\em dual problem},  and then arguing that the same rate applies to the primal suboptimality as well.}. In practice, SDCA is most effective  in situations where the number of examples ($n$) exceeds the number of features ($d$).  Since the dual of ERM is an $n$ dimensional problem, it makes intuitive sense to apply RCD to the dual. Indeed, RCD can be seen as a randomized decomposition strategy, reducing the $n$ dimensional problem to a sequence of (randomly generated) one-dimensional problems. 

However, if the number of features exceeds the number of examples, and especially when the difference is very large, RCD methods \cite{PCDM} have been found very attractive for solving the {\em primal  problem} (i.e., the ERM problem) directly. For instance, distributed variants of RCD, such as Hydra \cite{Hydra} and its accelerated cousin Hydra$^2$ \cite{Hydra2} have been successfully applied to solving problems with billions of features. 

Recently, a  variety of novel primal methods for ERM have been designed, including  SAG \cite{SAG}, SVRG \cite{SVRG}, S2GD \cite{S2GD}, proxSVRG \cite{proxSVRG}, mS2GD \cite{mS2GD}, SAGA \cite{SAGA}, MISO \cite{MISO} and S2CD~\cite{S2CD}. As SDCA, all these methods improve dramatically on SGD \cite{robbins1951} as a benchmark,  which they achieve by employing one of a number of variance-reduction strategies. However, these methods have essentially identical identical theoretical behavior to SDCA, including the property that these methods thrive in the data-laden domain (i.e., $n\gg d$). In this sense, in our comparison of primal vs dual RCD, these methods should be viewed as ``dual methods''.

\subsection{Contributions}
 In this paper we provide the first joint study of these two approaches---applying RCD to the primal vs dual problems---and we do so in the context of  L2-regularized ERM. First, we show through a rigorous theoretical analysis that for dense data, the intuition that the primal approach is better than the dual approach when $n\geq d$, and vice versa, is precisely correct. However, we show that for sparse data, this does not need to be the case:  primal RCD can significantly outperform dual RCD even if $d\ll n$, and vice versa, dual RCD can be much faster than primal RCD even if $n\ll d$. In particular, we identify that the face-off between primal and dual RCD boils down to the comparison of as single quantity associated with the data matrix and its transpose.  Moreover, we show that, surprisingly, a single sampling strategy minimizes both the (bound on the) number of iterations and the overall expected complexity of RCD. Note that the latter complexity measure takes into account also the average cost of the iterations, which depends on the structure and sparsity of the data, and on the sampling strategy employed.  We confirm our theoretical findings using extensive experiments with both synthetic and real data sets.

\section{Primal and Dual Formulations of ERM}

Let  $\bX\in \R^{d\times n}$ be a data matrix, with $n$ referring to the number of examples and $d$ to the number of features. With each example $\bX_{:j}\in \R^d$ we associate a loss function $\phi_j:\R\to \R$, and pick a regularization constant $\lambda>0$. The key problem of this paper is the L2-regularized ERM problem
\begin{equation}
\label{def:Primal}
\min_{w \in \R^d} \left[ P(w) \eqdef \frac{1}{n}\sum_{j=1}^n \phi_j(\langle \bX_{:j}, w \rangle) + \frac{\lambda}{2}\|w\|_2^2 \right],
\end{equation}
where $\langle\cdot, \cdot\rangle$ denotes the standard Euclidean inner product and $\|w\|_2\eqdef\sqrt{\langle w,w\rangle}$. We refer to \eqref{def:Primal} as the  {\em primal problem}. We assume throughout that the functions $\{\phi_j\}$ are convex and $\beta$-smooth: \begin{equation} \label{eq:phi_smooth}
\phi_j(s) + \phi_j'(s)t  \leq \phi_j(s + t) \leq \phi_j(s) + \phi_j'(s)t + \frac{\beta}{2}t^2, \qquad \text{for all} \qquad s,t\in \R.
\end{equation}
The {\em dual problem} of \eqref{def:Primal}  is
\begin{equation}\label{def:Dual} 
\max_{\alpha \in \R^n} \left[ D(\alpha) \eqdef - \frac{1}{2\lambda n^2}\left\| \bX\alpha\right\|_2^2 - \frac{1}{n}\sum_{j=1}^n \phi_j^*(-\alpha_j)  \right],
\end{equation} 
where $\phi_j^*:\R\to \R$ is the convex conjugate of $\phi_j$, defined by $
\phi_j^*(s) \eqdef \sup \{ st - \phi_j(t) \;:\; t \in \R\}.$ 
It is well known that \cite{SDCA, Quartz} that $P(w) \geq D(\alpha)$ for every pair $(w, \alpha) \in \R^d \times \R^n$ and $P(w^*) = D(\alpha^*)$. Moreover, the primal and dual optimal solutions, $w^*$ and $\alpha^*$, respectively, are unique, and satisfy the relations $w^* = \frac{1}{\lambda n} \bX \alpha^*$ and $\alpha_j^* = \phi_j'(\langle \bX_{:j}, w^* \rangle)$ for all  $j \in [n] \eqdef \{1, \dots, n \}$, which also uniquely characterize them.

\section{Primal and Dual RCD}

In its general ``arbitrary sampling''  form \cite{NSync}, RCD applied to the primal problem \eqref{def:Primal} has the form \begin{equation}\label{eq:RCD-primal-text}w_i^{k+1} \leftarrow w_i^k -  \frac{1}{u'_i}\nabla_i P(w^k)\quad \text{for} \quad i\in S_k, \qquad \qquad w_i^{k+1} \leftarrow w_k^k \quad \text{for} \quad i\notin S_k,\end{equation}
where $u'_1,\dots,u'_d>0$ are parameters of the method and $\nabla_i P(w) = \frac{1}{n}\sum_{j=1}^n \phi_j'(\langle \bX_{:j}, w\rangle ) \bX_{ij} + \lambda w_i$ is the $i$th partial derivative of $P$ at $w$. This update is performed for a random subset of the coordinates $i\in S_k\subseteq [d]$ chosen in an i.i.d.\ fashion according to some sampling $\hat{S}_P$. The parameters $u'_i$ are computed ahead of the iterative process and need to be selected carefully in order for the method to work \cite{NSync, ESO}. Specifically, one can set $u'_i \eqdef \frac{\beta}{n}u_i + \lambda$, where $u=(u_1,\dots,u_d)$ is chosen so as to satisfy the ESO (expected separable overapproximation) inequality \begin{equation}\label{eq:ESO_P}\bP \circ \bX \bX^\top \preceq \Diag(p\circ u),\end{equation} where $\bP$ is the $d\times d$ matrix with entries $\bP_{ij} = \Prob(i\in \hat{S}, j\in \hat{S})$, $p  = \Diag(\bP)\in \R^d$ and $\circ$ denotes the Hadamard (element-wise) product of matrices. The method is formally described as Algorithm~\ref{alg:nsync}.


\begin{algorithm}[H]
\caption{Primal RCD: NSync \cite{NSync}} \label{alg:nsync}
\begin{algorithmic}
\STATE \textbf{Input:} initial iterate $w^0 \in \R^d$; sampling $\hat{S}_P$; ESO parameters $u_1, \dots, u_d > 0$
\STATE \textbf{Initialize:} $z^0 = \bX^\top w^0$
\FOR{$k=0,1,\dots $}
\STATE Sample $S_k \subseteq [d]$ according to $\hat{S}_P$
\FOR{$i \in S_k$}
\STATE Compute $\Delta^{k}_i = -\frac{n}{\beta u_i + \lambda n}\left( \frac{1}{n}\sum_{j=1}^n \phi_j'(z_j^k) \bX_{ij} + \lambda w_i^k \right)$
\STATE Update $w_i^{k+1} = w_i^k + \Delta_i^k$
\ENDFOR
\FOR{$i\notin S_k$}
\STATE $w_i^{k+1} = w_i^k$
\ENDFOR
\STATE Update $z^{k+1} = z^k + \sum_{i \in S_k} \Delta_i^k\bX_{i:}^\top $
\ENDFOR
\end{algorithmic}
\end{algorithm}

When applying RCD to the dual problem \ref{def:Dual}, we can't proceed as above since the functions $\phi_j^*$ are not necessarily smooth, and hence we can't compute the partial derivatives of the dual objective. The standard approach here is to use a proximal variant of RCD \cite{PCDM}. In particular, Algorithm~\ref{alg:quartz} has been analyzed in \cite{Quartz}. Like Algorithm~\ref{alg:nsync}, Algorithm~\ref{alg:quartz} is also capable to work with an arbitrary sampling, which in this case is a random subset of $[n]$. The ESO parameters $v=(v_1,\dots,v_j)$ must in this case satisfy the ESO inequality \begin{equation}\label{eq:ESO_D}\bQ\circ \bX^\top \bX \preceq \Diag(q\circ v),\end{equation} where $\bQ$ is  an $n\times n$ matrix with entries $\bQ_{ij} = \Prob(i\in \hat{S}_D, j\in \hat{S}_D)$ and $q = \Diag(\bQ)\in \R^n$.

\begin{algorithm}[H]
\caption{Dual RCD: Quartz \cite{Quartz}} \label{alg:quartz}
\begin{algorithmic}
\STATE \textbf{Input:} initial dual variables $\alpha^0\in \R^n$, sampling $\hat{S}_D$; ESO parameters $v_1, \dots, v_n > 0$
\STATE \textbf{Initialize:} set $w^0 = \frac{1}{\lambda n} \bX \alpha^0 $
\FOR{$k=0, 1, \dots$}
\STATE Sample $S_k \subseteq [n]$ according to $\hat{S}_D$ 
\FOR{$j \in S_k$}
\STATE Compute $\Delta_j^k = \argmax_{h \in \R} \left\{ -\phi_j^*(-(\alpha_j^k + h)) - h\langle \bX_{:j}, w^k \rangle - \frac{v_j h^2}{2 \lambda n} \right\}$
\STATE Update  $\alpha_j^{k+1} = \alpha_j^k + \Delta_j^k $
\ENDFOR
\FOR{$j\notin S_k$}
\STATE $\alpha_j^{k+1} = \alpha_j^k$
\ENDFOR
\STATE Update $w^{k+1} = w^k + \frac{1}{\lambda n} \sum_{j \in S_k} \Delta_j^k \bX_{:j}$
\ENDFOR
\end{algorithmic}
\end{algorithm} 

 If we assume that $|\hat{S}_P|=1$ (resp.\ $|\hat{S}_D|=1$) with probability 1 (i.e., of the samplings are ``serial''), then it is trivial to observe that \eqref{eq:ESO_P} (resp.\ \eqref{eq:ESO_D}) holds with 
 \begin{equation}\label{eq:98s98hsd}u=\Diag(\bX \bX^\top) \qquad \text{(resp. } v=\Diag(\bX^\top \bX) \text{)}.\end{equation}
Easily computable expressions for $u$ (resp.\ $v$) for more complicated samplings can be found in \cite{ESO}.


\section{Iteration Complexity and Total Arithmetic Complexity}

In this section we give expressions for the total expected arithmetic complexity of the two algorithms.

\subsection{Number of iterations} Iteration complexity of Algorithms~\ref{alg:nsync} and \ref{alg:quartz} is described in the following theorem. 
We include a proof sketch in the appendix.

\begin{theorem} \label{thm:primal_and_dual} \textbf{(Complexity: Primal vs Dual RCD)}
Let $\{\phi_j\}$ be convex and $\beta$-smooth.
\\ (i) If $\hat{S}_P$ is proper (i.e., $p_i>0$ for all $i$), and $u$  satisfies \eqref{eq:ESO_P}, then iterates of primal RCD satisfy
\begin{equation}\label{eq:K_P}k \geq K_P = K_P( \hat{S}_P,\epsilon)\eqdef \max_{i \in [d]} \left( \frac{\beta u_i + \lambda n }{p_i \lambda n }\right) \log\left( \frac{\CP}{\epsilon} \right) \quad \Rightarrow \quad \E[P(w^k) - P(w^*)] \leq \epsilon,\end{equation}
where $\CP$ is a constant depending on $w^0$ and $w^*$.\\
(ii) If $\hat{S}_D$ is proper (i.e., $q_i>0$ for all $i$),  and $v$ satisfies \eqref{eq:ESO_D}, then  iterates of dual RCD satisfy
\begin{equation}\label{eq:K_D}k \geq K_D = K_D(\hat{S}_D,\epsilon) \eqdef \max_{j \in [n]} \left( \frac{\beta v_j + \lambda n }{q_j \lambda n }\right) \log\left( \frac{\CD}{\epsilon} \right) \quad \Rightarrow \quad \E[P(w^k) - P(w^*)] \leq \epsilon, \end{equation}
where $\CD$ is a constant depending on $w^0$ and $w^*$.

\end{theorem}

For the dual method a stronger guarantee can be established (see \cite{Quartz}): as soon as $k\geq L_D(\hat{S}_D,\epsilon)$, we have $\E[P(w^k)-D(\alpha^k)]\leq \epsilon$. Clearly, this stronger result implies the claim in part ii) of the above theorem.

\subsection{Average cost of a single iteration}

Let $\|\cdot\|_0$ be the number of nonzeros in a matrix/vector. It is easy to observe that the average  cost of a single iteration of  Algorithm~\ref{alg:nsync} is  \begin{equation}\label{eq:W_P}  W_P(\bX,\hat{S}_P)\eqdef \mathcal{O}\left(\E\left[ \sum_{i \in \hat{S}_P} \|\bX_{i:}\|_0\right]\right)= \mathcal{O}\left(\sum_{i=1}^d p_i \|\bX_{i:}\|_0\right),\end{equation}
and for Algorithm~\ref{alg:quartz} it is   \begin{equation}\label{eq:W_D}  W_D(\bX,\hat{S}_D)\eqdef \mathcal{O}\left(\E\left[ \sum_{j \in \hat{S}_D} \|\bX_{:j}\|_0\right]\right)= \mathcal{O}\left(\sum_{j=1}^n q_j \|\bX_{:j}\|_0\right).\end{equation}
We remark that the constant hidden in $\mathcal{O}$ may be larger for Algorithm~\ref{alg:nsync} than for Algorithm~\ref{alg:quartz}. The reason for this is that for Algorithm~\ref{alg:nsync} we compute the one-dimensional derivative $\phi_j'$ for every nonzero term in the sum, while for Algorithm~\ref{alg:quartz} we do this only once. Depending on the loss $\phi_j$, this may lead to slower iterations. For example, if $\phi_j$ isthe  logistic loss, experimentation shows that the constant is around 50. On the other hand, if $\phi_j$ is the squared loss, the constant is 1.
\subsection{Total complexity}

By combining the bounds on the number of iterations provided by Theorem~\ref{thm:primal_and_dual} with the formulas \eqref{eq:W_P} and \eqref{eq:W_D} for the cost of a single iteration we obtain the following expressions for the {\em total complexity} of  the two algorithms, where we ignore the logarithmic terms and drop the $\tilde{O}$ symbol:\begin{equation}\label{eq:T_P}  T_P = T_P(\bX,\hat{S}_P) \eqdef K_P W_P  \overset{\eqref{eq:K_P} + \eqref{eq:W_P}}{=}  \left( \max_{i \in [d]} \frac{\beta u_i + \lambda n }{p_i \lambda n }\right) \left(\sum_{i=1}^d p_i \|\bX_{i:}\|_0\right), \end{equation}  \begin{equation}\label{eq:T_D}  T_D = T_D(\bX,\hat{S}_D) \eqdef  K_D W_D  \overset{\eqref{eq:K_D} + \eqref{eq:W_D}}{=} \left( \max_{j \in [n]} \frac{\beta v_j + \lambda n }{q_j \lambda n }\right) \left(\sum_{j=1}^n q_j \|\bX_{:j}\|_0\right).\end{equation}

\section{Choosing a Sampling that Minimizes the Total Complexity}

In this section we  identify the {\em optimal sampling} in terms of the {\em total complexity.} This is different from previous results on {\em importance sampling}, which neglect to take into account the cost of the iterations \cite{NSync,Quartz,IProx-SDCA, ImportanceSrebro}. For simplicity, we shall only consider {\em serial} samplings, i.e., samplings which only pick a single coordinate at a  time. The situation is much more complicated with non-serial samplings where first importance sampling results  have only been derived recently \cite{ISM}.

\subsection{Uniform Sampling}

The simplest serial sampling is the {\em uniform sampling}: it selects every coordinate with the same probability, i.e. $p_i = 1/d, ~\forall i \in [d]$ and $q_j = 1/n, ~\forall j \in [n]$. In view of \eqref{eq:T_P}, \eqref{eq:T_D} and \eqref{eq:98s98hsd}, we get
\[T_P =\|\bX\|_0 \left( 1 + \frac{\beta}{\lambda n}\max_{i \in [d]} \|\bX_{i:}\|_2^2 \right)\] and \[T_D = \|\bX\|_0\left( 1 + \frac{\beta}{\lambda n}\max_{j \in [n]} \|\bX_{:j}\|_2^2 \right).\] 
We can now clearly see that whether $T_P\leq T_D$ or $T_P\geq T_D$ depends does not simply depend on $d$ vs $n$, but instead depends on the relative value of the quantities $\max_{i \in [d]} \|\bX_{i:}\|_2^2$  and $\max_{j \in [n]} \|\bX_{:j}\|_2^2$. Having said that, we shall not study these quantities in this paper. The reason for this is that for the cake of brevity, we shall instead focus on comparing the primal and dual RCD methods for optimal sampling which minimizes the total complexity, in which case we will obtain different quantities.

\subsection{Importance Sampling}

By {\em importance sampling} we mean the serial sampling $\hat{S}_P$ (resp.\ $\hat{S}_D$) which minimizes the bounds $K_P$ in \ref{eq:K_P} (resp.\ $K_D$ in \eqref{eq:K_D}). It can easily be seen (see also \cite{NSync}, \cite{Quartz}, \cite{IProx-SDCA}),  that importance sampling probabilities are given by \begin{equation} \label{def:importance_sampling}
p_i^* = \frac{\beta u_i + \lambda n}{\sum_{l} ( \beta u_l + \lambda n)} \qquad \mbox{and} \qquad q_j^* = \frac{\beta v_j + \lambda n}{\sum_{l} (\beta v_l + \lambda n)}.
\end{equation}  
On the other hand, one can observe that the average iteration cost of importance sampling may be larger than the average iteration cost of uniform serial sampling. Therefore, it is a natural question to ask, whether it is necessarily better. In view of \eqref{eq:T_P}, \eqref{eq:T_D} and \eqref{def:importance_sampling}, the total complexities for importance sampling are
\begin{equation} \label{eq:MMM}  T_P =\|\bX\|_0 + \frac{\beta}{\lambda n}\sum_{i=1}^d \|\bX_{i:}\|_0 \|\bX_{i:}\|_2^2 , \qquad  T_D = \|\bX\|_0 + \frac{\beta}{\lambda n}\sum_{j=1}^n \|\bX_{:j}\|_0\|\bX_{:j}\|_2^2. \end{equation}
Since a weighted average is smaller  than the maximum, the total complexity of both methods with importance sampling is always better than with uniform sampling. However, this does not mean that importance sampling is the sampling that minimizes total complexity. 

\subsection{Optimal Sampling}

The next theorem states that, in fact, importance sampling {\em does} minimize the total complexity.

\begin{theorem}\label{thm:runtime} The optimal  serial sampling (i.e., the serial sampling minimizing the total expected complexity $T_P$ (resp, $T_D$))  is the importance sampling \eqref{def:importance_sampling}.
\end{theorem}

\section{The Face-Off}

In this section we investigate the two quantities in \eqref{eq:MMM}, $T_P$  and $T_D$, measuring the total complexity of the two methods   as functions of the data $\bX$. Clearly, it is enough to focus on the quantities
\begin{equation} \label{eq:two_quantities}   \CP(\bX) \eqdef \sum_{i=1}^d \|\bX_{i:}\|_0 \| \bX_{i:}\|^2\qquad \text{and} \qquad \CD(\bX) \eqdef \sum_{j=1}^n \|\bX_{:j}\|_0 \| \bX_{:j}\|^2.\end{equation}
We shall ask questions such as: when is $\CP(\bX)$ larger/smaller than $\CD(\bX)$, and by how much. In this regard, it is useful to note that  $\CP(\bX) = \CD(\bX^\top)$.  Our first result gives tight lower and upper bounds on their ratio.

\begin{theorem}\label{thm:Bound_Any_Data} For any $\bX\in \R^{d \times n}$ with no zero rows or columns,  we have the bounds $\|\bX\|_F^2 \leq \CP(\bX) \leq n\|\bX\|_F^2$ and $\|\bX\|_F^2\leq \CD(\bX) \leq d\|\bX\|_F^2$. It follows that $ 1/d\leq \CP(\bX) / \CD(\bX) \leq n.$ Moreover, all these bounds are tight.
\end{theorem}

Since $\CP(\bX)$ (resp. $\CD(\bX)$) can dominate the expression \eqref{eq:T_P} (resp.\ \eqref{eq:T_D}) for total complexity, it follows that, depending on the data matrix $\bX$, {\em the primal method can be up to $d$ times faster than the dual method, and up to $n$ times slower than the dual method.}

\subsection{Random  Data and Dense Data}

Assume now that the entries of $\bX$ are chosen in an i.i.d.\ manner from some distribution with mean $\mu$ and variance $\sigma^2$. While this is not a realistic scenario, it will help us build intuition about what we can expect the quantities $\CP(\bX)$ and $\CD(\bX)$ to look like. A simple calculation reveals that $\E[\CP(\bX)] = dn\sigma^2 + dn^2\mu^2$, and $\E[\CD(\bX)] = dn\sigma^2 + nd^2\mu^2$. Hence, \[\E[\CP(\bX)]\leq \E[\CD(\bX)]\] precisely when $n\leq d$, which means that  the primal method is better when $n < d$ and the dual method is better when $n>d$. 

If $\bX$ is a dense deterministic matrix ($\bX_{ij}\neq 0$ for all $i,j$), then $\CP(\bX) = n \|\bX\|_F^2$ and $\CD(\bX)=d \|\bX\|_F^2$, and we reach the same conclusion as for random data: everything boils down to $d$ vs $n$.

\subsection{Binary Data} \label{sec:binary}

In this part we identify a class of data matrices for which one can have $\CP\leq \CD$ even if $d\ll n$. This class is by no means exhaustive, and serves as an example which we use to illustrate the phenomenon.

Let $\mathbb{B}^{d\times n}$ denote the set of $d\times n$ matrices $\bX$ with (signed) binary elements, i.e., with $\bX_{ij} \in \{-1,0,1\}$ for all $i,j$. For $\bX \in \mathbb{B}^{d\times n}$, the expressions in \eqref{eq:two_quantities} can be also written in the form $\CP(\bX) = \sum_{i=1}^d \|\bX_{i:}\|_0^2$ and $\CD(\bX) = \sum_{j=1}^n \|\bX_{:j}\|_0^2$.
By $\mathbb{B}_{\neq 0}^{d\times n}$ we denote the set of all matrices in $\mathbb{B}^{d\times n}$ with  nonzero columns and rows.

For positive integers $a,b$ we write $\bar{a}_b \eqdef b\left\lfloor\frac{a}{b}\right\rfloor$ (i.e., $a$ rounded down to the closest multiple of $b$). Further, we write \[R(\alpha,d,n) \eqdef U(\alpha,d,n)/ L(\alpha,n),\]  where 
\[L(\alpha,n) \eqdef \frac{1}{n} (\bar{\alpha}_n^2 + (\alpha - \bar{\alpha}_n)(2\bar{\alpha}_n + n))\] 
 and  \[U(\alpha,d,n) \eqdef (d + 1)\overline{(\alpha - n)}_{d-1} + n - 1 + (\alpha - n + 1 - \overline{(\alpha - n)}_{d-1})^2.\]

The following is a refinement  of Theorem~\ref{thm:Bound_Any_Data} for binary matrices of fixed cardinality $\alpha$.

\begin{theorem} \label{lem:bound_ratio}
	For all $\bX \in \mathbb{B}^{d\times n}_{\neq 0}$ with $\alpha = \|\bX\|_0$ we have the bounds  $1/R(\alpha,n,d) \leq \CP(\bX) / \CD(\bX) \leq R(\alpha, d, n)$. Moreover, these bounds are tight.
\end{theorem}

The above theorem follows from Lemma~\ref{lem:general_binary_maxmin}, which we formulate and prove in the Appendix. This lemma establishes formulas for the minimum and maximum of $\CD$ and $\CP$, subject to the constraint $\|\bX\|_0=\alpha$, in terms of the functions $L$ and $U$. Further, as we show in Lemma~\ref{lem:bound_Rabc} in the Appendix, if $d \geq n$ and $\alpha \geq n^2 + 3n$, then $R(\alpha,d,n) \leq 1$. Likewise, if $n \geq d$ and $\alpha \geq d^2 + 3d$, then $R(\alpha,n,d) \leq 1$. Combined with Theorem~\ref{lem:bound_ratio}, this has an interesting consequence, spelled out in the next theorem and its corollary.

\begin{theorem} \label{thm:98y9s8s} Let $\bX \in \mathbb{B}^{d\times n}_{\neq 0}$.
If $d \geq n$ and   $\| \bX \|_0 \geq n^2 + 3n$, then $\CP(\bX) \leq \CD(\bX)$. By symmetry, if $n \geq d$ and $\| \bX\|_0 \geq d^2 + 3d$, then $\CD(\bX) \leq \CP(\bX)$. 
\end{theorem}

This result says that for binary data, and $d\geq n$, the primal method is better than the dual method even for non-dense data, as long as the  the data is ``dense enough''.	Observe that as long as $d \geq n^2 + 3n$, 	all matrices $\bX \in \mathbb{B}^{d\times n}_{\neq 0}$ satisfy $\|\bX\|_0 \geq d \geq n^2 + 3n \geq n$. This leads to the following corollary.
	
\begin{corollary}
If $d \geq n^2 + 3n$, then for all $\bX \in \mathbb{B}^{d\times n}_{\neq 0}$ we have $\CP(\bX) \leq \CD(\bX)$. By symmetry, if $n \geq d^2 + 3d$, then for all  $\bX \in \mathbb{B}^{d\times n}_{\neq 0}$ we have $\CD(\bX) \leq \CP(\bX)$.\end{corollary}

In  words, the corollary states that for binary data where the number of features ($d$) is large enough in comparison with the number of examples ($n$), the primal method will be always better. On the other hand, if $n$ is large enough, the dual method will be always better. This behavior can be observed in Figure~\ref{fig:alpha}. For large enough $d$, all the values $R(\alpha,d,n)$ are below 1.


\begin{figure}[H]
	\centering
	\begin{subfigure}{.30\linewidth}
		\label{fig:plotR1000}
		\centering
		\includegraphics[width = \linewidth]{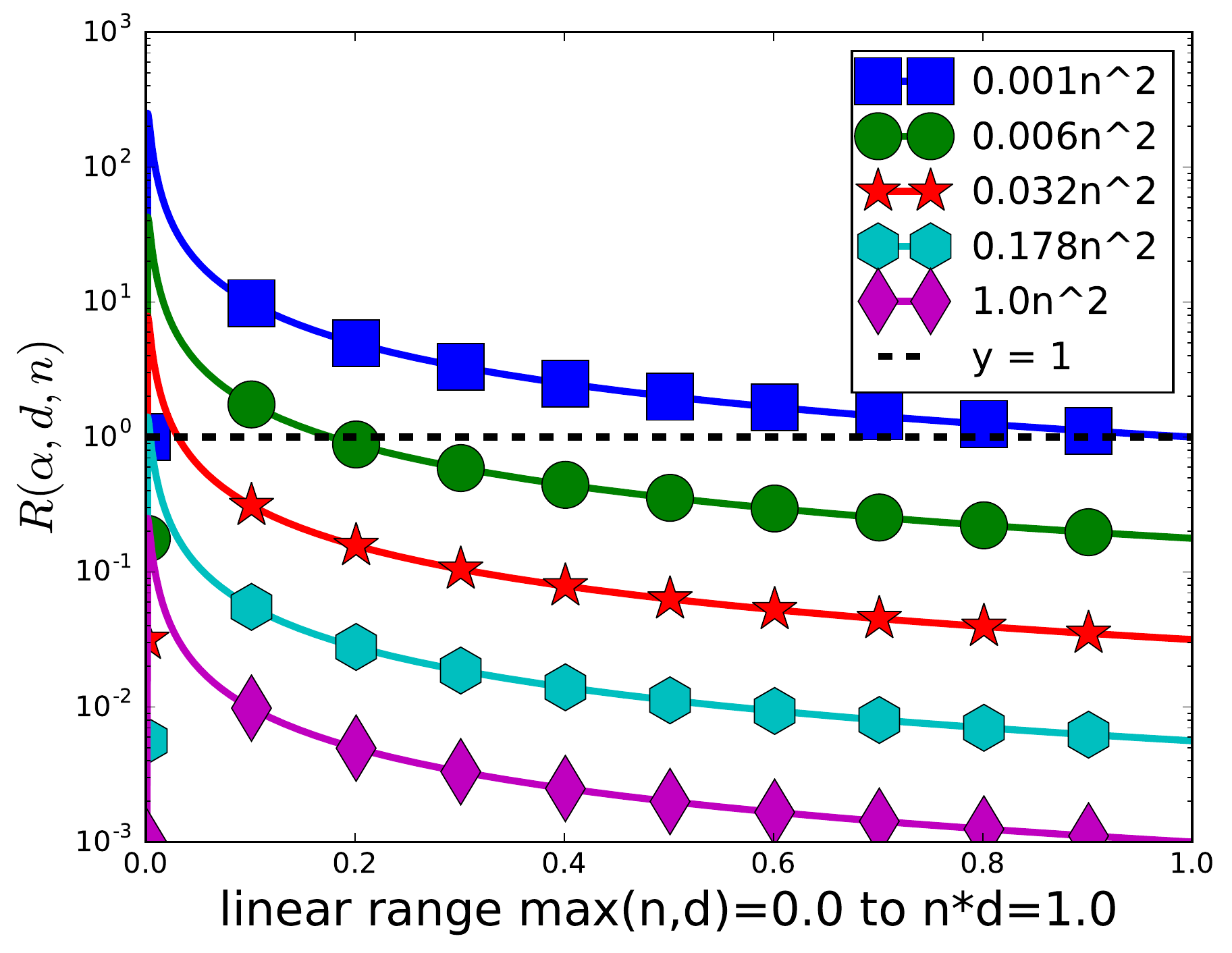}
	\end{subfigure}
	\caption{The value $R(\alpha,d,n)$ plotted for $n=10^3$, $n\leq d\leq n^2$ and $\max\{d,n\} \leq \alpha \leq nd$.}
	\label{fig:alpha}
\end{figure}

\section{Experiments}

We conducted experiments on both real and synthetic data. The problem we were interested in is a standard logistic regression with an L2-regularizer, i.e., \[P(w) = \frac{1}{n}\sum_{j=1}^n \log(1 + \exp(-y_j \langle \bX_{:j}, w \rangle)) + \frac{\lambda}{2}\|w\|_2^2.\] In all our experiments we used $\lambda = 1/n$ and we normalized all the entries of $\bX$ by the average column norm.
 Note that for logistic loss there is no closed form solution for $\Delta_j^k$ in Algorithm~\ref{alg:quartz}. Therefore we use a variant of Algorithm~\ref{alg:quartz} where $\Delta_j^k = \eta (\phi_j'(\langle \bX_{:j}, w \rangle) + \alpha_j^k)$ with the step size $\eta$ defined as $\eta = \min_{j \in [n]}(q_j \lambda n)/(\beta v_j + \lambda n)$. This variant has the same convergence rate guarantees as Algorithm~\ref{alg:quartz} and does not require exact minimization. Details can be found in \cite{Quartz}.

We plot the training error against the number of passes through the data. The number of passes is calculated according to the number of visited nonzero entries in the matrix $\bX$. One pass means that we look at $\|\bX\|_0$ nonzero entries of $\bX$, but not necessarily all of them. We look at the problems from the perspective of the primal approach. The same could be done symmetrically for the dual approach.

\subsection{General Data}

We look at the matrices which give worst-case bounds for general matrices (Theorem~\ref{thm:Bound_Any_Data}) and their empirical properties for different choices of $d$ and $n$. The corresponding figures are Figure~\ref{subfig:general1} and \ref{subfig:general2}. For a square dataset, we can observe a large speed-up. For large $n$ we can observe, that the theory holds and the primal method is still faster, but because of numerical issues (we need very small and very large numbers in matrix) and the fact that the optimal value is very close to an "initial guess" of the algorithm, the difference in speed is more difficult to observe.

\begin{figure}
\centering
	\begin{subfigure}{\mywidth \textwidth} 
		\includegraphics[width = \textwidth]{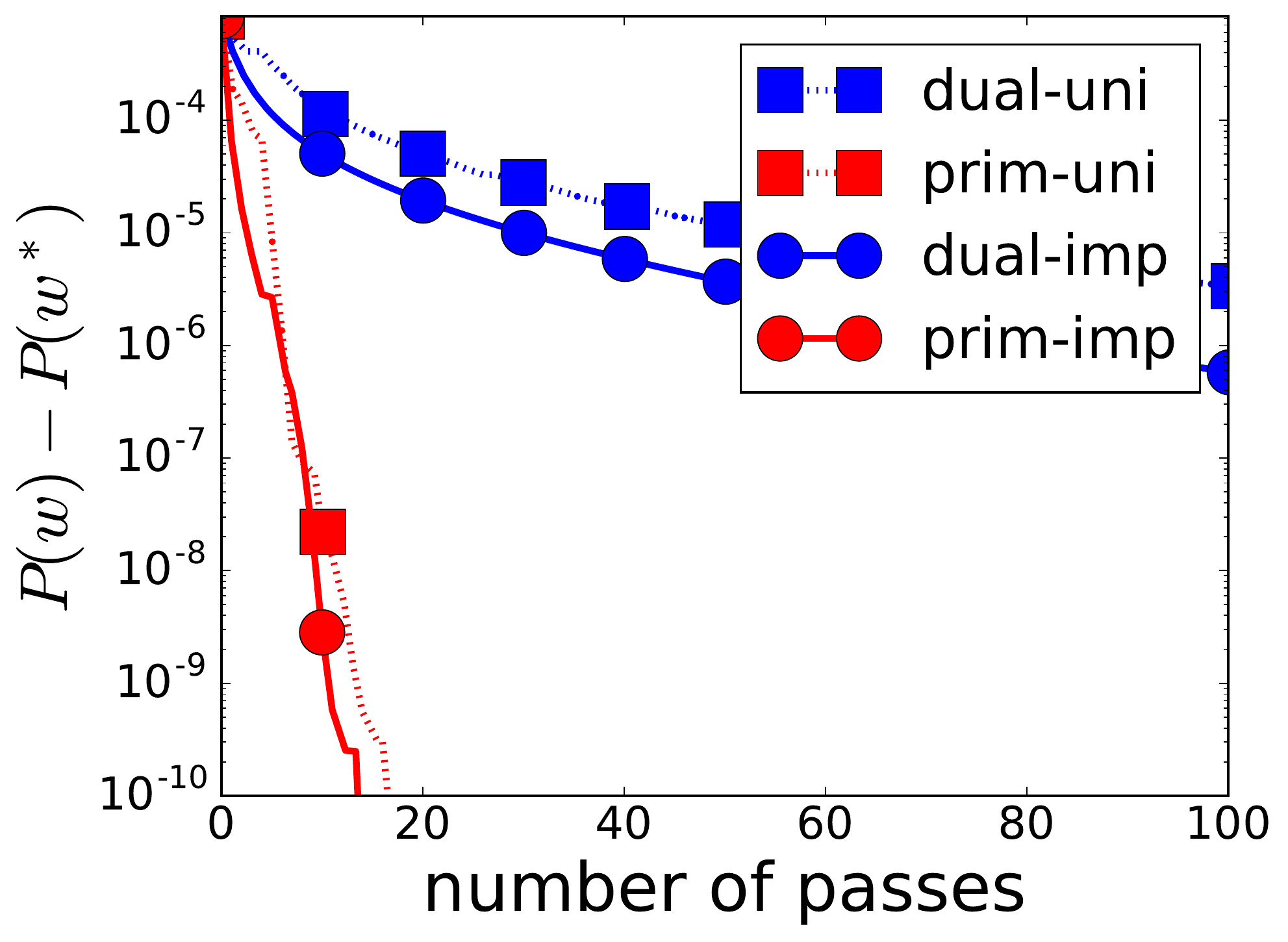}
		\caption{$1,000 \times 1,000$}
		\label{subfig:general1}
	\end{subfigure}
	\begin{subfigure}{\mywidth \textwidth} 
		\includegraphics[width = \textwidth]{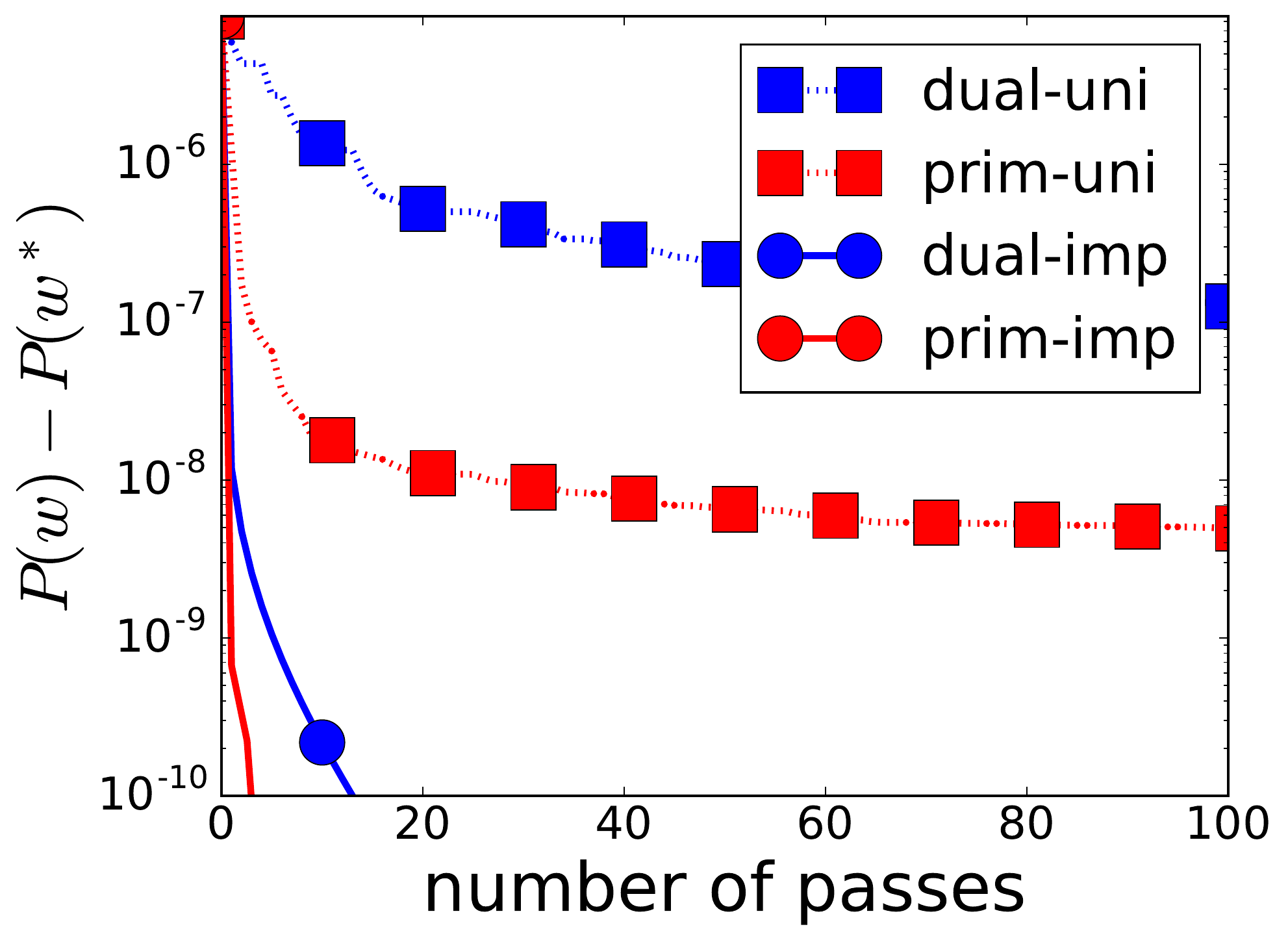}
		\caption{$100 \times 100,000$}
		\label{subfig:general2}
	\end{subfigure}
	\begin{subfigure}{\mywidth \textwidth} 
		\includegraphics[width = \textwidth]{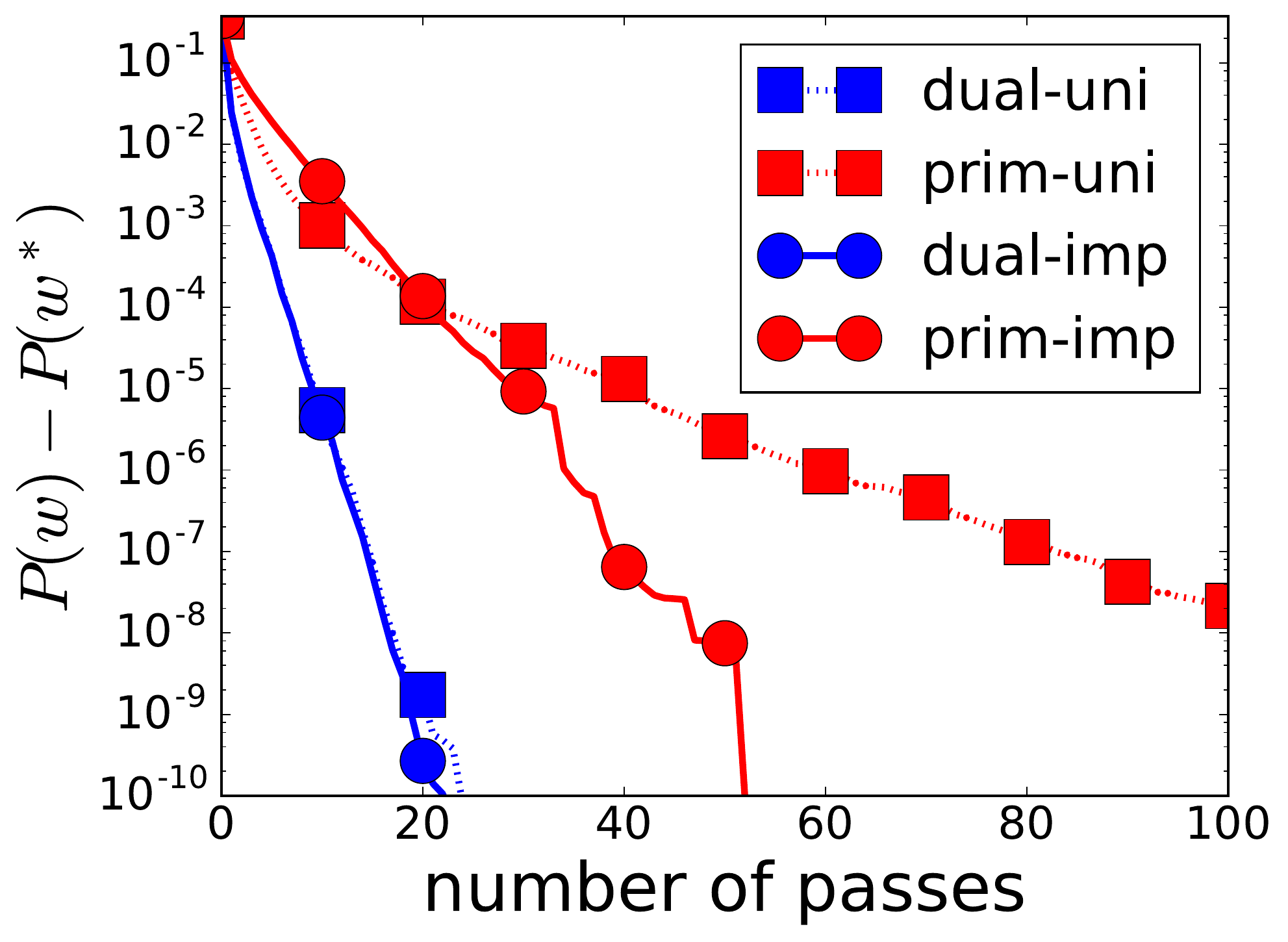}
		\caption{news dataset}
		\label{subfig:news}
	\end{subfigure}
	\begin{subfigure}{\mywidth \textwidth} 
		\includegraphics[width = \textwidth]{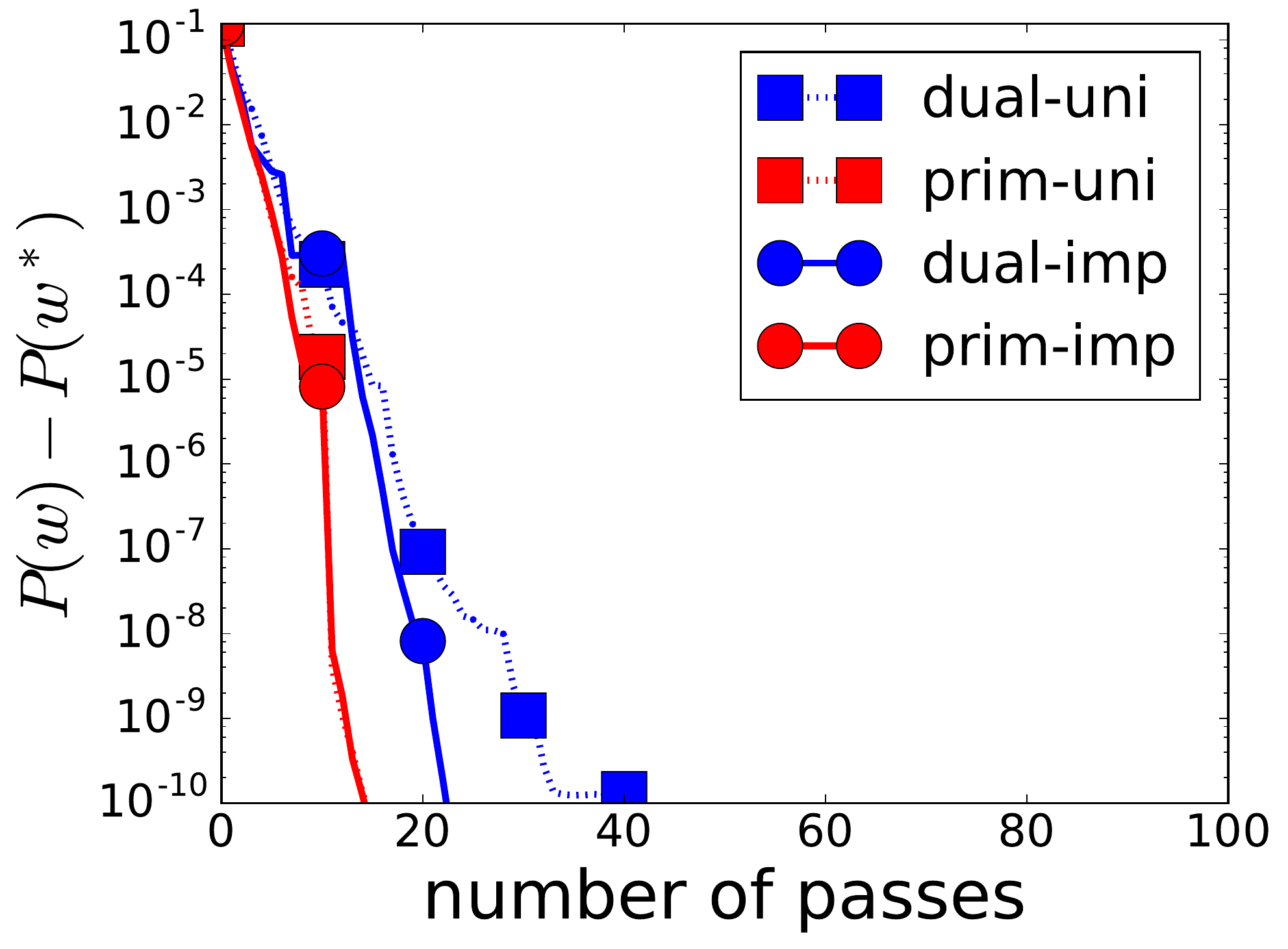}
		\caption{leukemia dataset}
		\label{subfig:leukemia}
	\end{subfigure}
	\caption{Testing the worst case for general matrices and real datasets}
	\label{fig:general_real} 
\end{figure}

\subsection{Synthetic Binary Data}

 We looked at matrices with all entries in  $\{a, -a, 0\}$ for some $a \neq 0$. We fixed the number of features to be $d = 100$ and we varied the number of examples $n$ and the sparsity level $\alpha=\|\bX\|_0$. For each triplet $[d,n,\alpha]$ we produced the worst-case matrix for dual RCD according to the developed theory. The results are in Figure~\ref*{tab:table_experiments}.

\begin{figure}[]
	\begin{tabular}{m{0.02\textwidth}   >{\centering\arraybackslash} m{0.28\textwidth}
			>{\centering\arraybackslash} m{0.28\textwidth}  >{\centering\arraybackslash} m{0.28\textwidth} }
		& nnz $ \sim 1\% $ &   nnz $ \sim 10\% $ & nnz $ 100\% $ \\ 
		\rotatebox[origin=c]{90}{$ n = 100 $} & \includegraphics[width=\mywidth\textwidth]{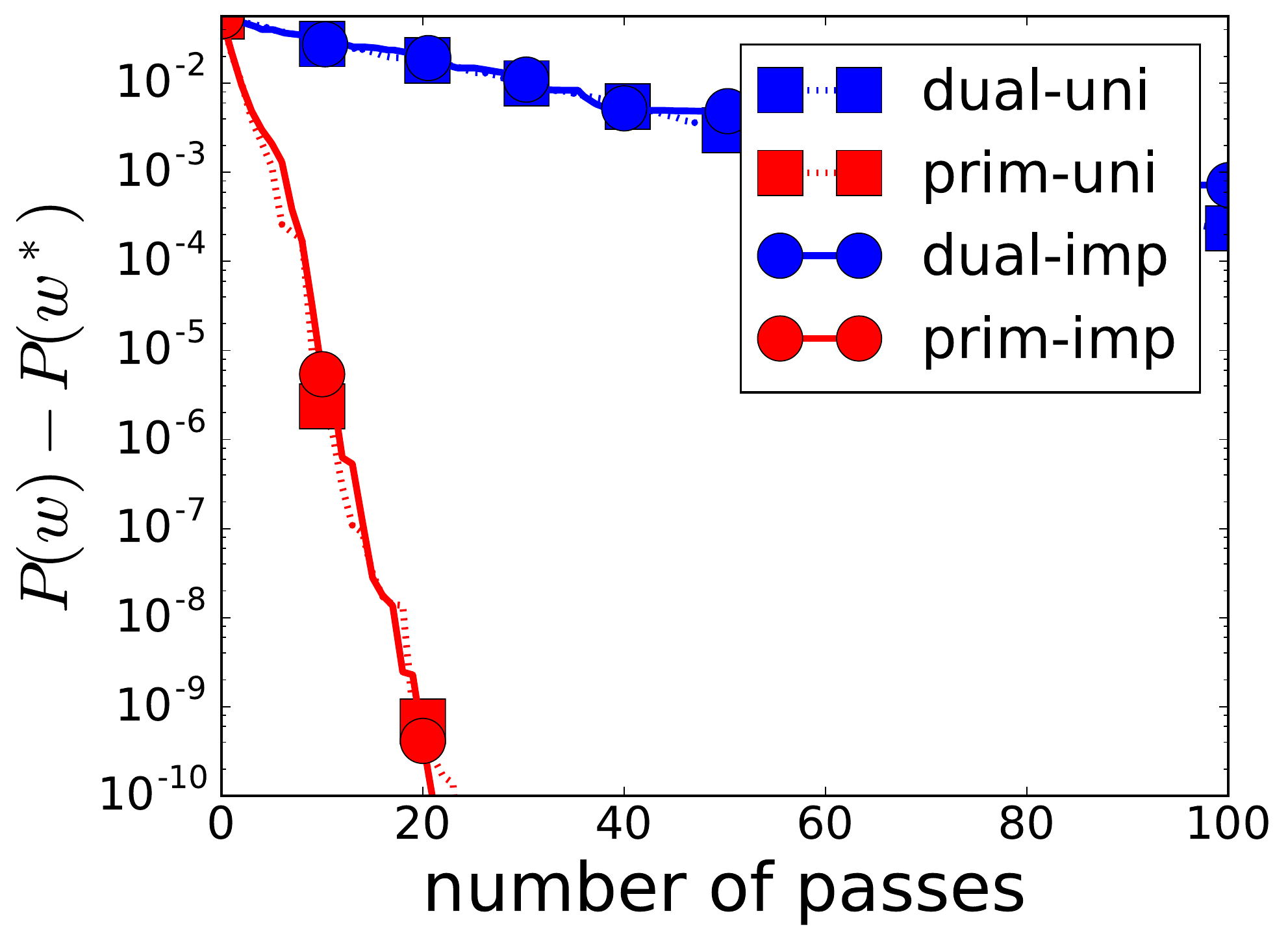}& \includegraphics[width=\mywidth\textwidth]{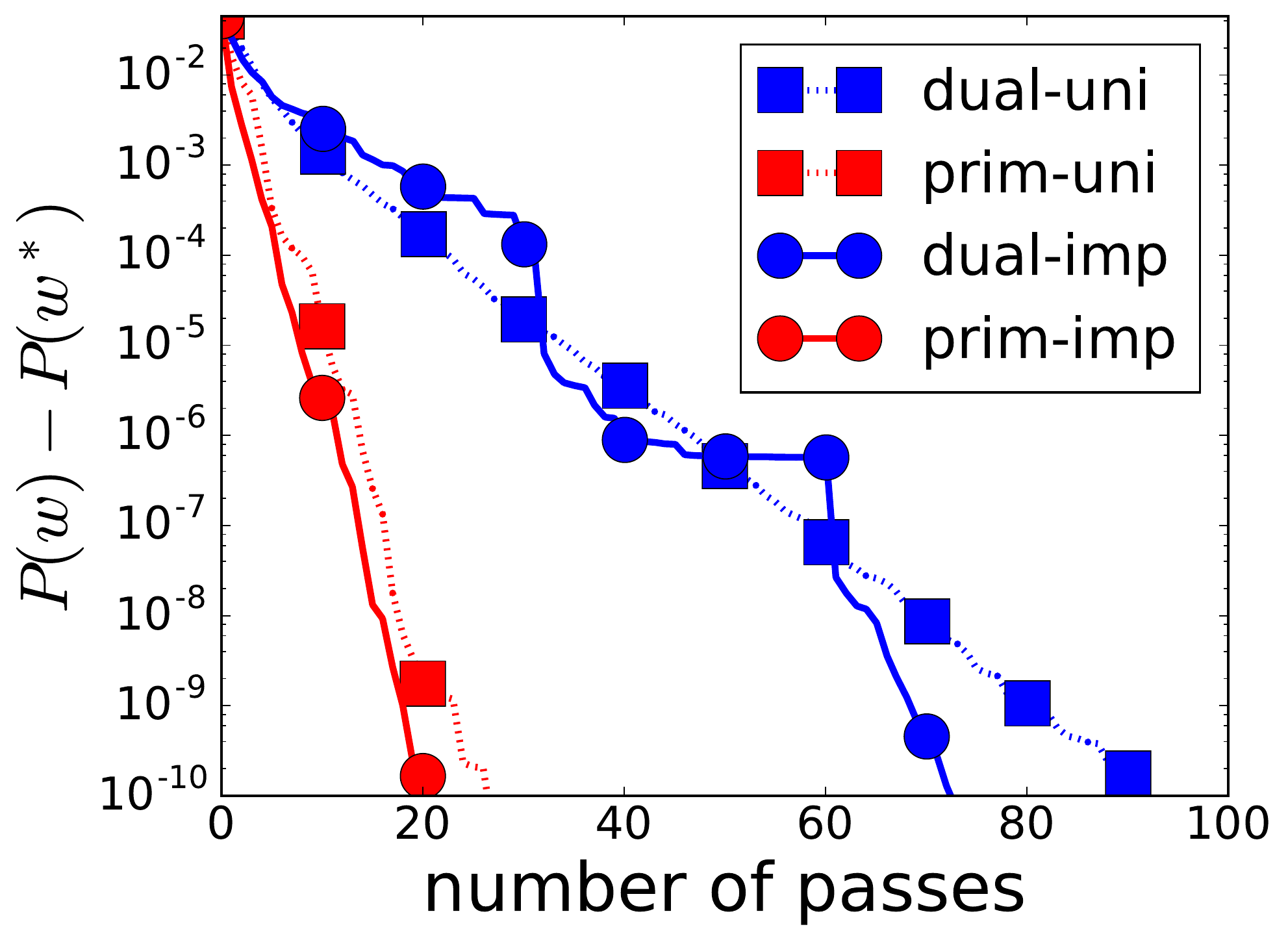}& \includegraphics[width=\mywidth\textwidth]{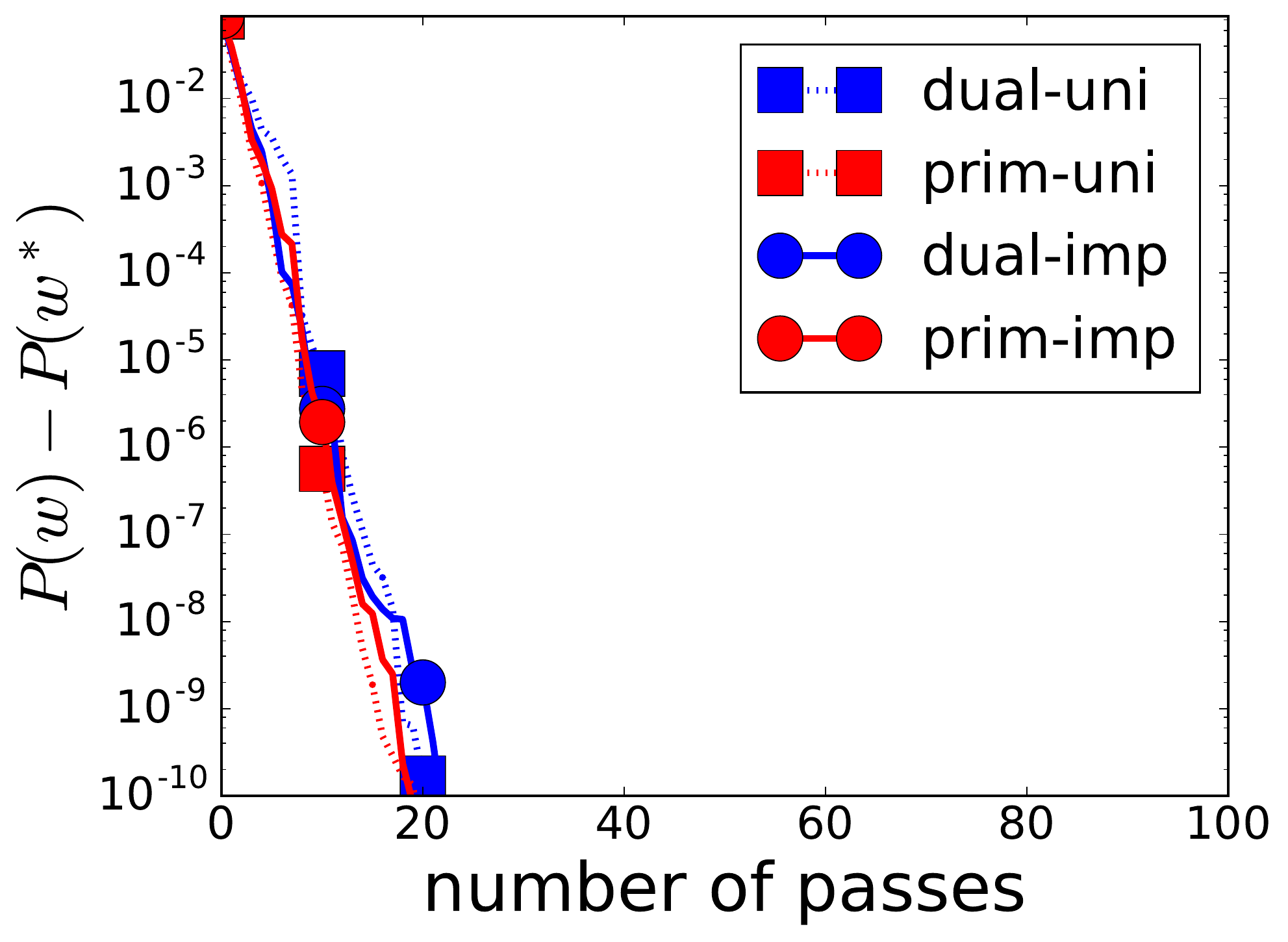}\\
		\rotatebox[origin=c]{90}{$ n = 1,000 $} & \includegraphics[width=\mywidth\textwidth]{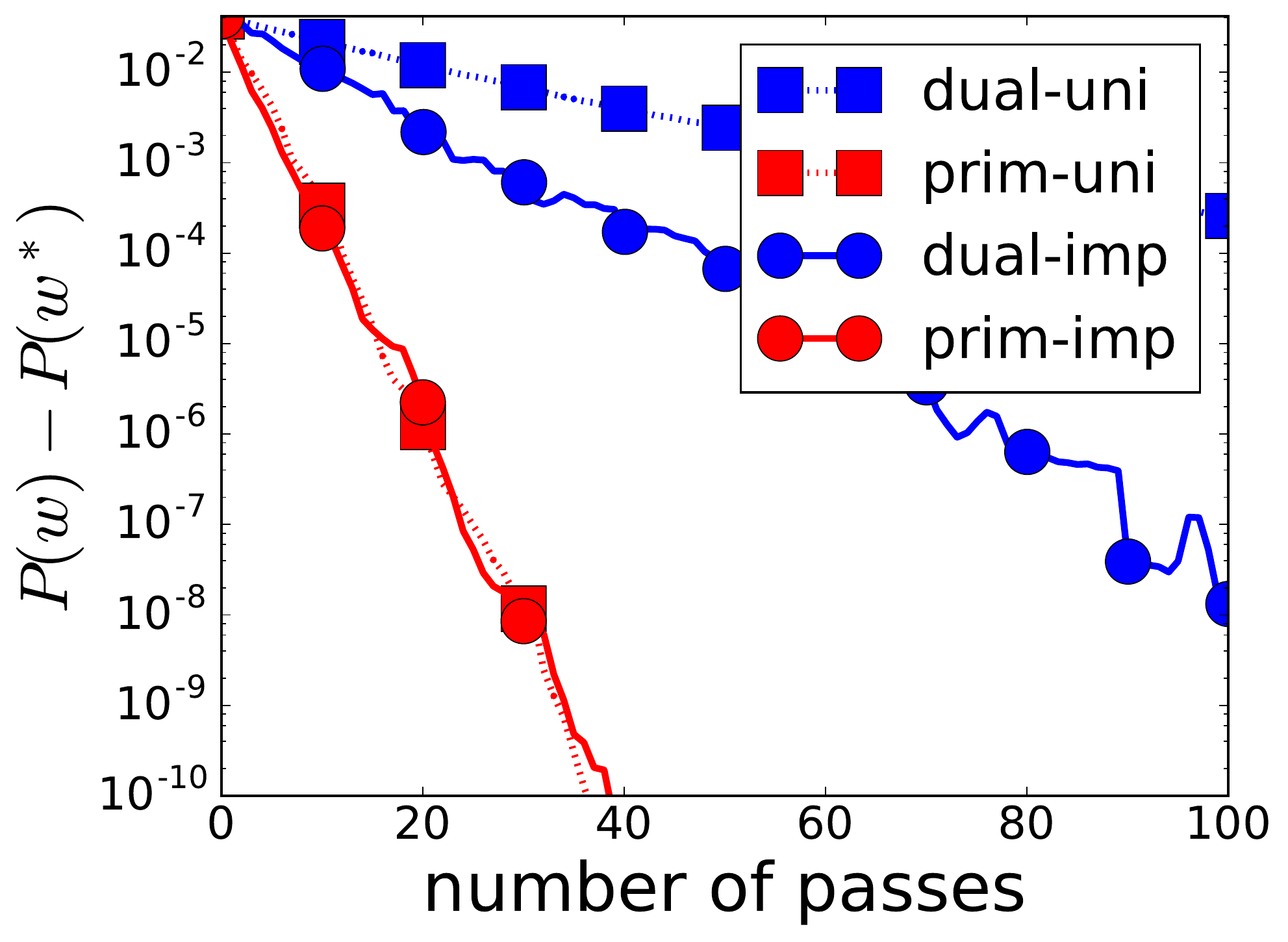}& \includegraphics[width=\mywidth\textwidth]{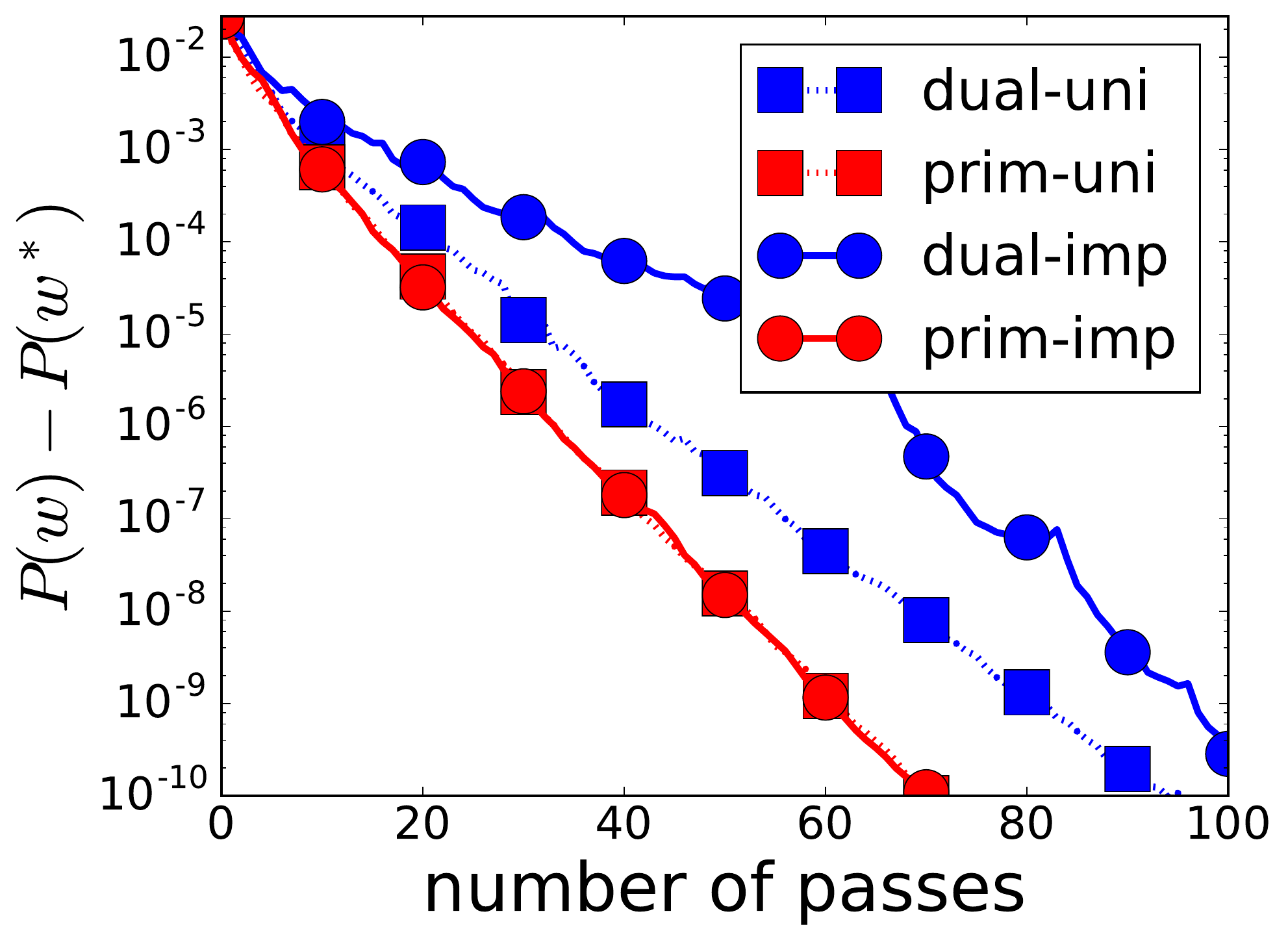}& \includegraphics[width=\mywidth\textwidth]{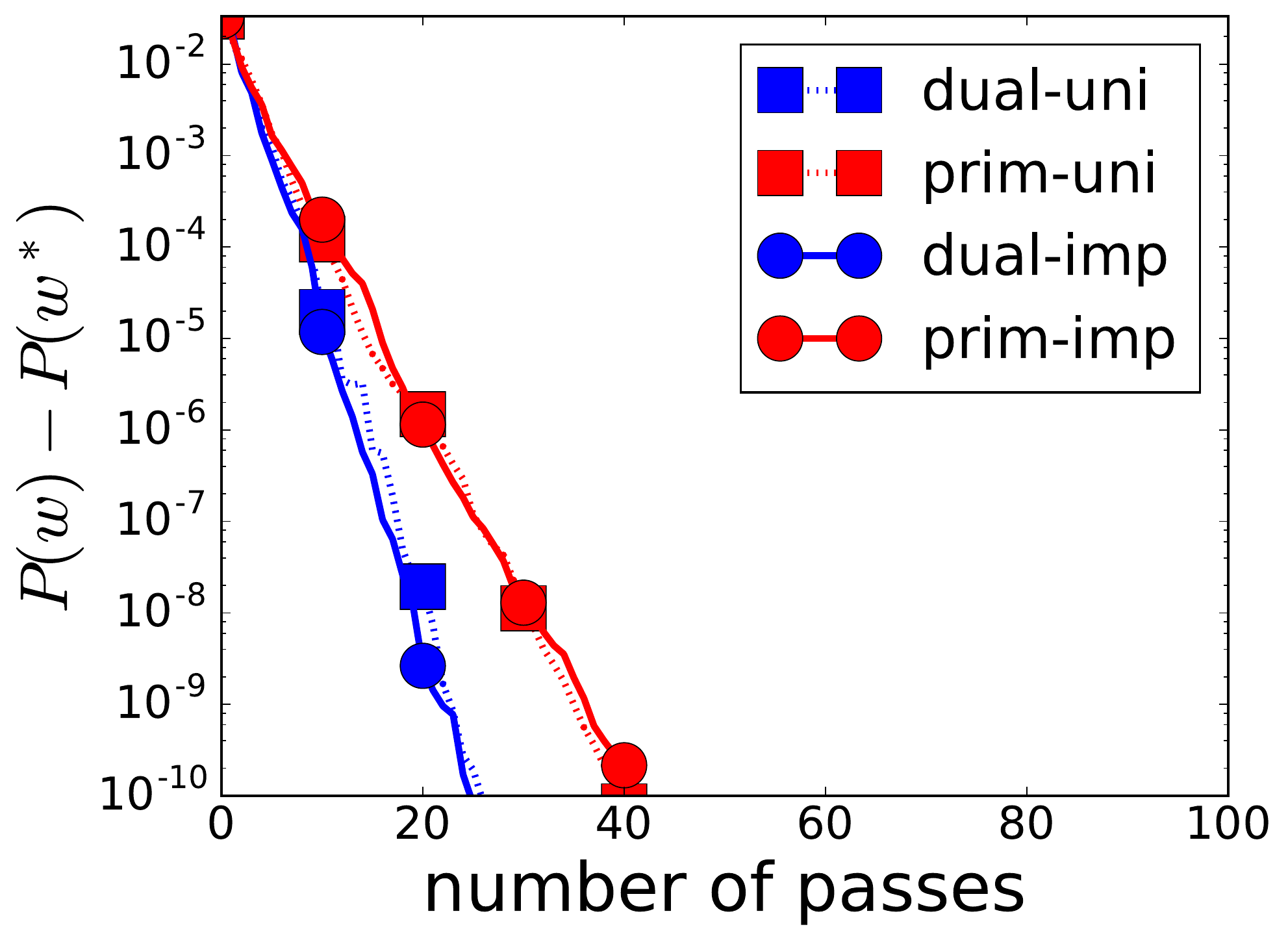}\\
		\rotatebox[origin=c]{90}{$ n = 10,000 $} & \includegraphics[width=\mywidth\textwidth]{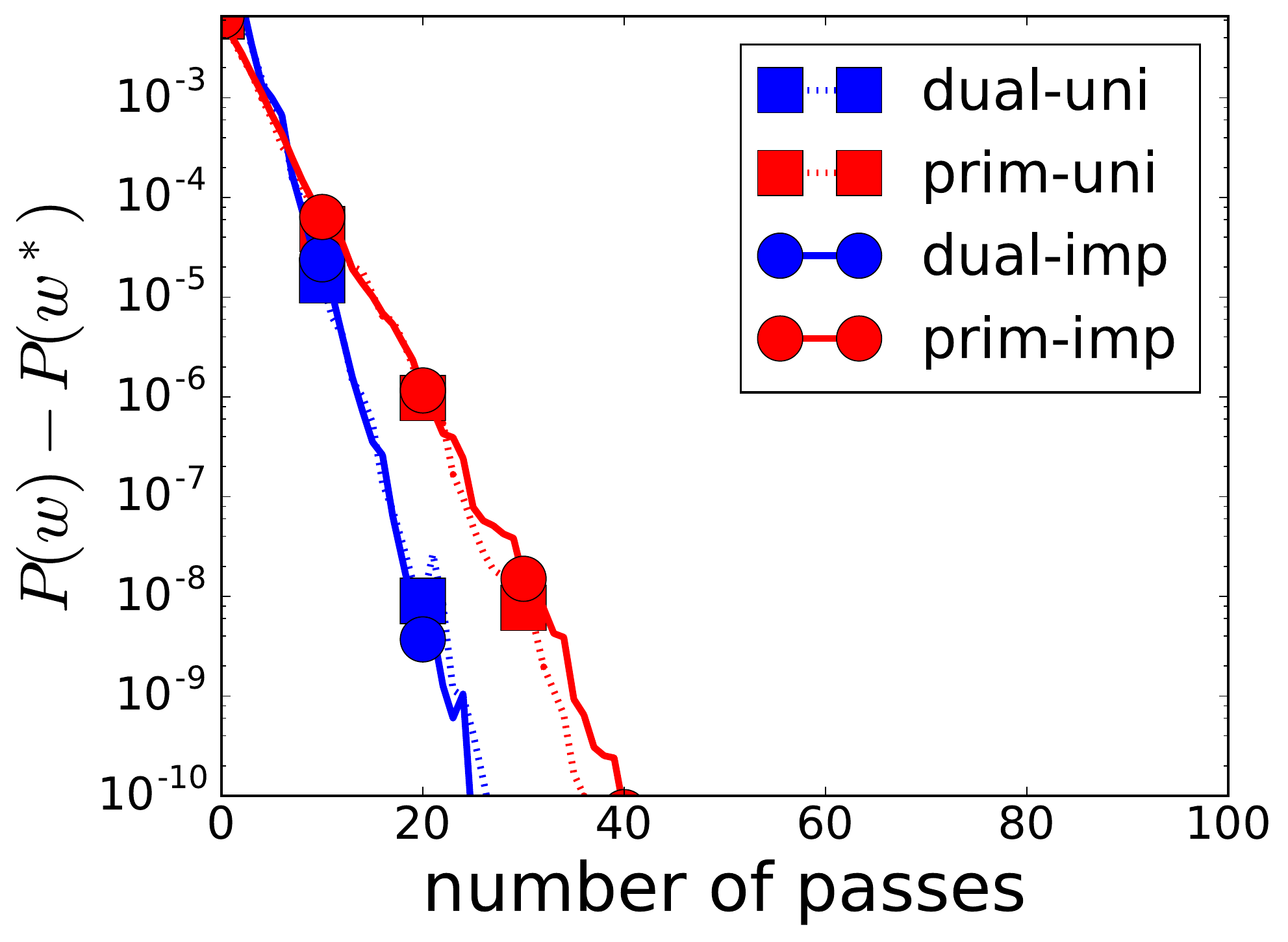}& \includegraphics[width=\mywidth\textwidth]{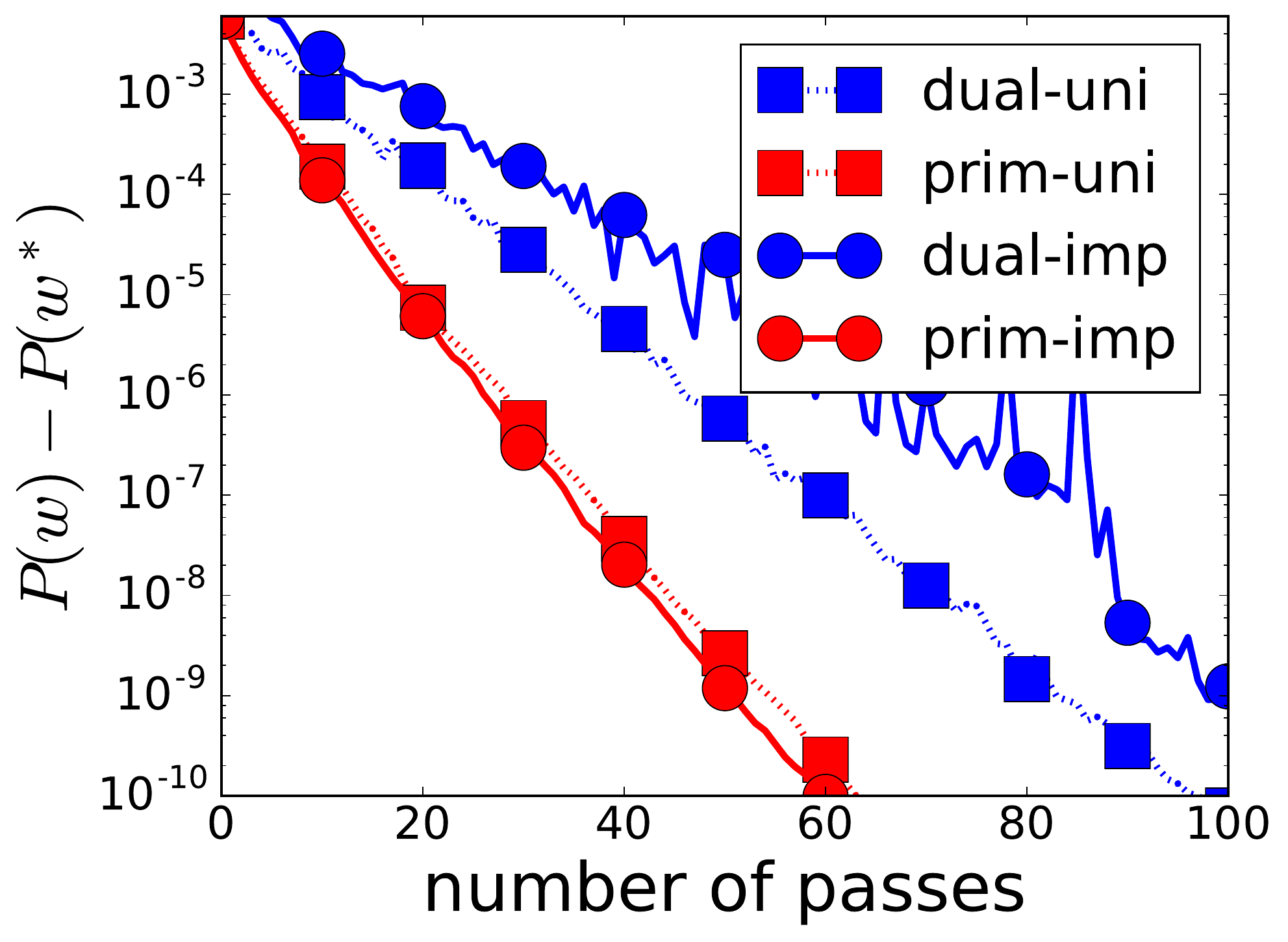}& \includegraphics[width=\mywidth\textwidth]{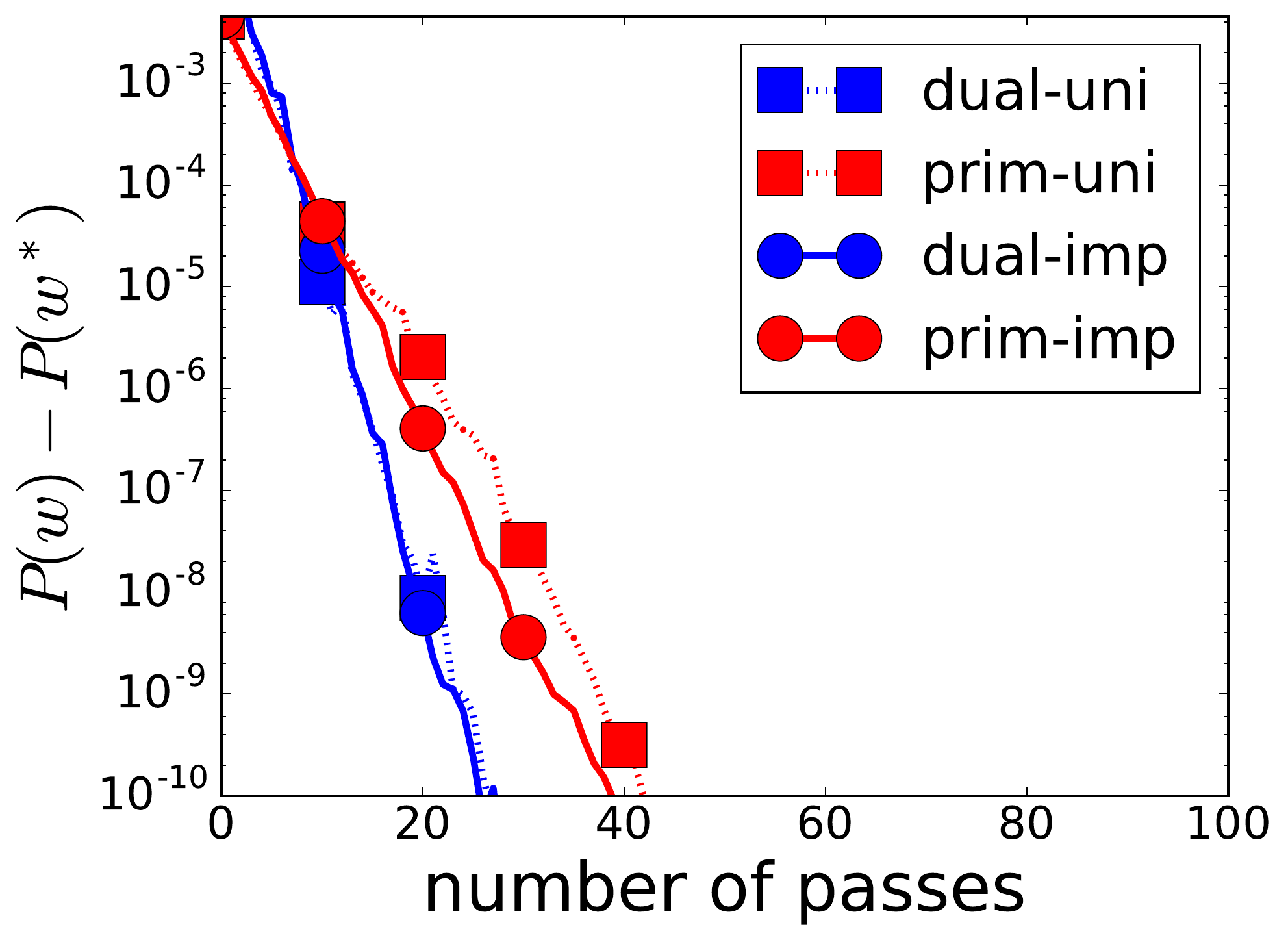}
	\end{tabular} 
		\caption{Worst-case experiments with various dimensions and sparsity levels for $d = 100$}
	\label{tab:table_experiments}
\end{figure}

\subsection{Real Data}

We used two real datasets to showcase our theory: news and leukemia\footnote{both datasets are available from https://www.csie.ntu.edu.tw/~cjlin/libsvmtools/datasets/}. The news dataset in Figure~\ref{subfig:news} is a nice example of our theory in practice. As shown in Table~\ref{tab:alltable} we have $d \gg n$, but the dual method is empirically faster than the primal one. The reason is simple: the news dataset uses a bag of words representation of news articles. If we look at the distribution of features (words), there are many words which appear just very rarely and there are words commonly used in many articles. The features have therefore a very skewed distribution of their nonzero entries. On the other hand, the examples have quite uniform distribution, as the number of distinct words in an article acts nicely. This distribution of nonzero entries highly favors the dual approach, as shown in the theory. The leukemia dataset in Figure~\ref{subfig:leukemia} is a fully dense dataset and  $d \gg n$. Therefore, as our theoretical analysis shows, the primal approach should be better. The ratio between the runtimes is not very large, as the constant $\|\bX\|_0$ is of similar order as the additional term in the computation of the true runtime. The empirical speedup in Figures~\ref{subfig:news} and Figures~\ref{subfig:leukemia}  matches the theoretical predictions from Table~\ref{tab:alltable}.

\begin{table}[]
	{
	\begin{center}
	\begin{tabular}{l r r r r l l cc}
		dataset & $d$ & $n$ & density & $\|\bX\|_0$ & $\CP$ & $\CD$ &  $T_P/T_D$ \\ 
		\hline news & 1,355,191 & 19,996 & 0.03\% & 9,097,916 & $3 \times 10^7$ & $9 \times 10^6$ & 2.0  \\
		leukemia & 7,129 & 38 & 100.00\% & 270,902 & $1\times 10^7$ & $2 \times 10^9$ & 0.5  
		\\\end{tabular}
		\end{center}
		}
	\caption{Details on the datasets used in the experiments}
	\label{tab:alltable}		
\end{table}

\section{Conclusions and Extensions} We have shown that the question whether RCD should be applied to the primal or the dual problem depends on the structure of the training dataset. For dense data, this simply boils down to whether we have more data or parameters, which is intuitively appealing. We have shown, both theoretically, and through experiments with synthetic and real datasets, that contrary to what seems to be a popular belief, primal RCD can outperform dual RCD even if $n\gg d$.

 In order to focus on the main message, we have chosen to present our results for simple (as opposed to ``accelerated'') variants of RCD. However, our results can be naturally extended to accelerated  variants of RCD, such as APPROX \cite{APPROX}, ASDCA~\cite{ASDCA}, APCG~\cite{APCG}, ALPHA~\cite{ALPHA} and SPDC~\cite{SPDC}. 
 
 Likewise, for simplicity, we focused on serial sampling (i.e., sampling a single coordinate). However, it is possible to use our approach to gain insights into the performance of primal vs dual RCD for arbitrary sampling strategies \cite{NSync, Quartz, ALPHA, ESO}.

{
\bibliography{final}
}

\clearpage

\section*{APPENDIX}

\section*{Proof of Theorem~\ref{thm:primal_and_dual}}


We say that $P \in \mathcal{C}^1(\bM)$, if \[P(w + h) \leq P(w) + \langle \nabla P(w), h\rangle + \frac{1}{2}h^\top \bM h, \quad \forall w,h \in \R^d.\]
For three vectors $a,b,c \in \R^n$ we define $\langle a,b\rangle_c \eqdef \sum_{i=1}^d a_ib_ic_i$ and $\|a\|_c^2 \eqdef \langle a, a \rangle_c = \sum_{i=1}^d c_i a_i^2 .$ Also, let for $\emptyset \neq S\subseteq [d]$ and $h\in \R^d$, we write $h_{S} \eqdef \sum_{i \in S} h_i e_i$, where $e_i$ is the $i$-th coordinate  vector (i.e., standard basis vector) in $\R^d$.

We will need the following two lemmas.
\begin{lemma} \label{lem:C1M}The primal objective $P$ satisfies $P \in {\cal C}^1(\bM)$, where $\bM =\lambda \bI + \frac{\beta}{n}\bX \bX^\top$. 
\end{lemma}
\begin{proof}
\begin{align*}
P(w + h) &\stackrel{\eqref{def:Primal}}{=} \frac{1}{n}\sum_{i=1}^n \phi_i (\langle \bX_{:i}, w \rangle  + \langle \bX_{:i} ,h \rangle) + \frac{\lambda}{2}\|w + h\|^2 \\
&\stackrel{\eqref{eq:phi_smooth}}{\leq} \frac{1}{n}\sum_{i=1}^n \left[ \phi_i(\langle \bX_{:i}, w \rangle) + \phi_i'(\langle \bX_{:i}, w \rangle)\cdot \langle \bX_{:i} , h \rangle + \frac{\beta}{2}\langle \bX_{:i}, h \rangle^2 \right] + \frac{\lambda}{2}\|w\|^2 + \lambda \langle w, h \rangle + \frac{\lambda}{2}\|h\|^2 \\ 
&= \frac{1}{n}\sum_{i=1}^n \phi_i(\langle \bX_{:i}, w \rangle) + \frac{\lambda}{2}\|w\|^2 + \left\langle \frac{1}{n}\sum_{i=1}^n \phi_i'(\langle \bX_{:i}, w \rangle) \bX_{:i} + \lambda w  ~,~ h \right\rangle \\ 
&\quad + \frac{1}{2} h^\top \left( \frac{\beta}{n} \sum_{i=1}^n \bX_{:i} (\bX_{:i})^\top + \lambda \bI \right) h \\
&= P(w) + \langle \nabla P(w), h \rangle + \frac{1}{2} h^\top \bM h.
\end{align*}
\end{proof}

\begin{lemma}\label{lem:ESO_equiv}
	If $P \in \mathcal{C}^1(\bM)$ and $u'\in \R^d$ is such that $\bP \circ \bM \preceq \Diag(p \circ u')$, then \[\E[P(w + h_{[\hat{S}_P]})] \leq P(w) + \langle \nabla P(w), h \rangle_p + \frac{1}{2}\|h\|_{p \circ u'}^2.\]
\end{lemma}
\begin{proof}
	See \cite{ESO}, Section 3.
\end{proof}

We can now proceed to the proof of Theorem ~\ref{thm:primal_and_dual}.

First, note that \[\bP \circ \bM = \lambda \mbox{Diag}(p) + \frac{\beta}{n} (\bP \circ \bX\bX^\top) \preceq \lambda \mbox{Diag}(p) + \frac{\beta}{n}\mbox{Diag}(p \circ u) \]
with $u$ defined as in \eqref{eq:ESO_P}.  We now separately establish the two complexity results; (i) for primal RCD and (ii) for dual RCD.

(i) The proof is a consequence of the proof of the main theorem of \cite{NSync}. Assumption 1 from \cite{NSync} holds with $w_i := \lambda + \frac{\beta}{n}u_i$ (Lemma~\ref{lem:C1M} \& Lemma~\ref{lem:ESO_equiv}) and Assumption 2 from \cite{NSync} holds with standard Euclidean  norm and $\gamma := \lambda$. We follow the proof all the way to the bound \[\E[P(w^k) - P(w^*)] \leq (1 - \mu)^k (P(w^0) - P(w^*))\]
which holds for $\mu$ defined by \[ \mu := \frac{\lambda}{\max_i \frac{n\lambda + \beta u_i}{n p_i}}\] by direct substitution of the quantities. The result follows by standard arguments. Note that $C_P = P(w^0) - P(w^*).$

(ii) The proof is a direct consequence of the proof of the main theorem of \cite{Quartz}, using the fact that $ P(w^k) - P(w^*) \leq P(w^k) - D(\alpha^k)$, as the weak duality holds. Note that $C_D = P(w^0) - D(\alpha^0).$

\section*{Proof of Theorem~\ref{thm:runtime}}

The proofs for Algorithm~\ref{alg:nsync} and Algorithm~\ref{alg:quartz} are analogous, and hence we will  establish the result for Algorithm~\ref{alg:nsync} only. For brevity, denote $s_i = \beta u_i + \lambda n$. We aim to solve the  optimization problem: \begin{equation} \label{eq:runimp1}
p^* \quad \leftarrow \quad \argmin_{p \in \R^d_+ ~:~ \sum_i p_i = 1} \quad T_P \overset{\eqref{eq:T_P}}{=}  \left( \max_{i \in [d]} \frac{s_i}{ p_i \lambda n} \right) \cdot \sum_{i=1}^d p_i \|\bX_{i:}\|_0.
\end{equation}
First observe, that the problem is homogeneous in $p$, i.e., if $p$ is optimal, also $c p$ will be optimal for $c >0$, as the solution will be the same. Using this argument, we can remove the constraint $\sum_i p_i = 1$. Also, we can remove the multiplicative factor $1/(\lambda n) $ from the denominator as it does not change the $\argmin$. Hence we get the simpler problem \begin{equation} \label{eq:runimp2}
p^* \quad \leftarrow \quad \argmin_{p \in \R^d_+} \quad \left[ \left( \max_{i \in [d]} \frac{s_i}{p_i} \right) \cdot \sum_{i=1}^d p_i \|\bX_{i:}\|_0 \right].
\end{equation}
Now choose optimal $p$ and assume that there exist $j,k$ such that $s_j/p_j < s_k/p_k$. By a small decrease in $p_j$, we will still have  $s_j/p_j \leq s_k/p_k$, and hence the term $\max_i s_i/p_i$ stays unchanged. However, the term $\sum_i p_i \|\bX_{i:}\|_0$ decreased. This means that the optimal sampling must satisfy $s_i/p_i = const$ for all $i$. However, this is precisely the importance sampling.


\section*{Proof of Theorem~\ref{thm:Bound_Any_Data}}

By assumption, all rows and columns of $\bX$ are nonzero. Therefore,  $1 \leq \|\bX_{i:}\|_0\leq n$ and $1\leq \|\bX_{:j}\|_0\leq d$, and the bounds on $\CP$ and $\CD$ follow by applying this to \eqref{eq:two_quantities}. The bounds for the ratio follow immediately by combining the previous bounds. It remains to establish tightness. For $a,b,c\in \R$, let $\bX(a,b,c) \in \R^{d \times n}$ be the matrix defined as follows: \[\bX_{ij}(a,b,c) = \begin{cases}
a \quad&i \neq 1 \wedge j=1   \\
b \quad&i = 1 \wedge j \neq 1 \\
c \quad&i = 1 \wedge j = 1 \\
0 \quad&\text{otherwise}.
\end{cases}\]
Notice that $\bX(a,b,c)$ does not have any zero rows nor columns as long as $a,b,c$ are nonzero. Since $\CP(\bX(a,b,c)) =   (d-1)a^2+n(n-1)b^2 + nc^2$ and $\CD(\bX(a,b,c)) = d(d-1)a^2 + (n-1)b^2 + dc^2$, one readily sees that
\[\lim_{\substack{b\to 0 \\ c\to 0}} \frac{\CP(\bX(a,b,c))}{\CD(\bX(a,b,c))} = \frac{1}{d} \qquad \text{and} \qquad  \lim_{\substack{a\to 0 \\ c\to 0}} \frac{\CP(\bX(a,b,c))}{\CD(\bX(a,b,c))} = n .\]


\newpage
\section*{Proof of Theorem~\ref{lem:bound_ratio}}

We first need a lemma.

\begin{lemma} \label{lem:general_binary_maxmin} Let $\alpha$ be an integer satisfying $\max\{d,n\}\leq \alpha \leq dn$ and let $L$ and $U$ be the functions defined in Section~\ref{sec:binary}. We have the following identities:

\begin{eqnarray}L(\alpha,n) &=&   \min_{\bX \in \BB} \{ \CD(\bX) \;:\; \|\bX\|_0 = \alpha \} \label{eq:4ID-1} \\
L(\alpha,d) &=& \min_{\bX \in \BB}   \{\CP(\bX) \;:\; \|\bX\|_0 = \alpha \} \label{eq:4ID-2}\\
 U(\alpha, d, n) &=& \max_{\bX \in \BB} \{\CD(\bX) \;:\; \|\bX\|_0 = \alpha\} \label{eq:4ID-3}\\
 U(\alpha, n, d) &=& \max_{\bX \in \BB} \{\CP(\bX) \;:\; \|\bX\|_0 = \alpha \} . \label{eq:4ID-4}
 \end{eqnarray}
\end{lemma}
\begin{proof}
	Let $\bX \in \BB$ be an arbitrary matrix and let $ \omega = (\omega_1, \dots, \omega_n)$, where $\omega_j \eqdef \|\bX_{:j}\|_0$. Let $\alpha = \|\bX\|_0 = \sum_j \omega_j$.	Observe that $\CD(\bX) = \sum_{j=1}^n \|\bX_{:j}\|_0^2=  \|\omega\|_2^2$.

\begin{enumerate}
\item[(i)] We shall first establish \eqref{eq:4ID-1}. Assume that the exist two columns $j,k$ of $\bX$, such that $\omega_j + 2 \leq \omega_k$, i.e., their difference in the number of nonzeros is at least 2. Because $\omega_k > \omega_j$, there has to exist a row which has a nonzero entry in the $k$-th column and a zero entry in the $j$-th column. Let $\bX'$ be the matrix obtained from $\bX$ by switching these two entries. Note that $\CP(\bX) = \CP(\bX')$. However, we have \[\CD(\bX) - \CD(\bX') = \omega_j^2 + \omega_k^2 - (\omega_j+1)^2 - (\omega_k - 1)^2 = 2 \omega_k - 2\omega_j - 2 > 0.\] It follows that while there exist two such columns, the minimum is not achieved. So, we only need to consider matrices $\bX$ for which there exists integer $a$ such that $\omega_j = a$ or $\omega_j = a+1$ for every $j$. Let $b = |\{j ~:~ \omega_j = a\}|$. 


We can now without loss of generality assume that $0 \leq b \leq n-1$. Indeed, we can do this is because the choices $b=0$ and $b=n$ lead to the same matrices, and hence by focusing on $b=0$ we have not removed any matrices from consideration. With simple calculations we get  \[\alpha = ba + (n-b)(a+1) = n(a + 1) - b.\]
Note that $\alpha + b$ is a multiple of $n$. It follows that $b = n - \alpha + \bar{\alpha}_n$ and $a = \bar{\alpha}_n / n$. Up to the ordering of the columns (which does not affect $\CD(\bX)$) we have just one candidate $\bX$, therefore it has to be the minimizer of $\CD$. Finally, we can easily calculate the minimum as \begin{align*} \sum_{j=1}^n \omega_j^2 &= ba^2 + (n-b)(a+1)^2 = (n-\alpha + \bar{\alpha}_n)\left(\frac{\bar{\alpha}_n}{n} \right)^2 + (\alpha-\bar{\alpha}_n)\left( \frac{\bar{\alpha}_n}{n} + 1 \right)^2 \\
		&= \frac{1}{n}\left(\bar{\alpha}_n^2 + (\alpha - \bar{\alpha}_n)(2 \bar{\alpha}_n + n) \right)= L(\alpha,n).
	\end{align*}
	
\item[(ii)]	Claim \eqref{eq:4ID-2}  follows from part \eqref{eq:4ID-1} via symmetry: $\CP(\bX) = \CD(\bX^\top)$ and $\|\bX\|_0=\|\bX^\top\|_0$. 
	
\item[(iii)] We now establish claim \eqref{eq:4ID-3}. Assume that there exist a pair of columns $j,k$ such that $1 < \omega_j \leq \omega_k < d$. Let $\bX'$ be the matrix obtained from $\bX$ by  zeroing out an entry in the $j$-th column and putting a nonzero inside the $k$-th column. Then \[\CD(\bX') - \CD(\bX) = (\omega_j-1)^2 + (\omega_k + 1)^2 - \omega_j^2 - \omega_k^2 = 2\omega_k - 2\omega_j + 2 > 0.\]
	It follows that while there exist such a pair of columns, the maximum is not achieved. This condition leaves us with matrices $\bX$ where at most one column $j$ has $\omega_j$ {\em not} equal to 1 or $d$. 
	
	Formally, let $a = |\{j ~:~ \omega_j = d\}|$. Then we have $n-a-1$ columns with $1$ nonzero and 1 column with $b$ nonzeros, where $1 \leq b < d$. This is correct, as $b = d$ is the same as $b=1$ and $a$ being one more. We can compute $a$ and $b$ from the equation \begin{align*} (n - a - 1)\cdot 1 + 1 \cdot b + a\cdot d &= \alpha \\
		b + a(d-1) &= \alpha - n + 1
	\end{align*}
	as the only solution to the division with remainder of $\alpha - n + 1$ by $d-1$, with the difference that $b \in \{1, \dots, d-1\}$ instead of the standard $\{0, \dots, d-2\}$. We get 
	\begin{align*}
		a = \left\lfloor \frac{a - n}{d-1} \right\rfloor \qquad \text{and} \qquad b = \alpha - n + 1 - \overline{(\alpha - n)}_{d-1}.	
	\end{align*}
	The maximum can now be easily computed as follows: \begin{align*}
		\sum_{j=1}^n \omega_j^2 &= (n - a - 1) + b^2 + ad^2 \\
		&= n-\left\lfloor \frac{a - n}{d-1} \right\rfloor-1 + \left(\alpha - n + 1 - \overline{(\alpha - n)}_{d-1} \right)^2 + \left\lfloor \frac{a - n}{d-1} \right\rfloor d^2 \\
		& = U(\alpha,d,n).
	\end{align*}
\item[(iv)]	Again, claim \eqref{eq:4ID-4} follows from \eqref{eq:4ID-3} via symmetry.
	\end{enumerate}
\end{proof}

We can now proceed to the proof of the theorem.

The quantity is the ratio between the maximal value of $\CP$ and the minimal value of $\CD$, we have to show that there exists a matrix $\bX$ such that this is achieved. Assume we have a matrix $\bX$ which has the maximal $\CP$. In the proof of Lemma~\ref{lem:general_binary_maxmin} we showed, that by switching entries in $\bX$ we can get the minimal value of $\CD$ without changing $\CP$. Therefore we can achieve maximal $\CP$ and minimal $\CD$ at the same time. Analogically for every other case.

\section*{Proof of Theorem~\ref{thm:98y9s8s} }

As shown in the main text, the theorem follows from the following lemma. Hence, we only need to prove the lemma.

\begin{lemma}\label{lem:bound_Rabc} If $d \geq n$ and $\alpha \geq n^2 + 3n$, then $R(\alpha,d,n) \leq 1$. If $n \geq d$ and $\alpha \geq d^2 + 3d$, then $R(\alpha,n,d) \leq 1$.
\end{lemma}
\begin{proof} We focus on the first part, the second follows in an analogous way.  
Using the two assumptions, we have $\alpha(n^2 + 3n) + n^3 \leq \alpha^2 + dn^2 $. By adding $n^2+n$ to the right hand side and after reshuffling, we obtain the inequality
\[n\left[(n+1)(\alpha - d) + d-1 + n^2\right] \leq  (\alpha - n)^2.\]
For positive scalars $a,b>0$, we have the trivial estimates  $ a - b \leq \bar{a}_b \eqdef b \lfloor \frac{a}{b}\rfloor \leq a$. We use them to bound four expressions: \begin{align*}
(\alpha - d) &\geq \overline{(\alpha - d)}_{n-1} \\
n^2 &\geq (\alpha - d + 1 - \overline{(\alpha - d)}_{n-1})^2 \\
\bar{\alpha}_n^2 &\geq (\alpha - n)^2 \\
(\alpha - \bar{\alpha}_n)(2\bar{\alpha}_n + n) &\geq 0 
\end{align*}
Using these bounds one-by-one we get the result
	\begin{eqnarray*}
	n\left[(n+1)(\alpha - d) + d-1 + n^2\right] & \leq & (\alpha - n)^2 \\
	n\left[(n+1)\overline{(\alpha - d)}_{n-1} + d-1 + n^2\right] & \leq & (\alpha - n)^2 \\
	n\left[(n+1)\overline{(\alpha - d)}_{n-1} + d-1 + (\alpha - d + 1 - \overline{(\alpha - d)}_{n-1})^2\right] & \leq & (\alpha - n)^2 \\
	n\left[(n+1)\overline{(\alpha - d)}_{n-1} + d-1 + (\alpha - d + 1 - \overline{(\alpha - d)}_{n-1})^2\right] & \leq & \bar{\alpha}_n^2 \\
	n \left[ (n+1)\overline{(\alpha - d)}_{n-1} + d - 1 + (\alpha - d + 1 - \overline{(\alpha - d)}_{n-1})^2 \right] &\leq & \bar{\alpha}_n^2 + (\alpha - \bar{\alpha}_n)(2\bar{\alpha}_n + n) \\
	R(\alpha,d,n) &\leq & 1
	\end{eqnarray*}

\end{proof}

\end{document}